\RequirePackage{fix-cm}
\documentclass[11pt]{article}
\usepackage{graphicx,amsmath,amsfonts,amscd,amssymb,bm,cite,epsfig,epsf,url,color,algorithm,algorithmic,enumerate}
\usepackage{fullpage}
\usepackage[small,bf]{caption}
\usepackage{subfigure}
\usepackage{lscape}
\usepackage[space]{grffile}
\usepackage{epsfig,psfrag,amstext,subfigure,subeqnarray,nccfoots,color,multirow,bm}

\newtheorem{theorem}{Theorem}[section]
\newtheorem{lemma}[theorem]{Lemma}
\newtheorem{claim}[theorem]{Claim}
\newtheorem{corollary}[theorem]{Corollary}
\newtheorem{proposition}[theorem]{Proposition}
\newtheorem{definition}[theorem]{Definition}

\newtheorem{assumption}[theorem]{Assumption}

\DeclareMathOperator{\argmin}{\mbox{argmin}}

\def\argmin{\mathop{\rm argmin}}

\newcommand{\dist}{{\rm dist}}

\numberwithin{equation}{section}

\def \endprf{\hfill {\vrule height6pt width6pt depth0pt}\medskip}
\newenvironment{proof}{\noindent {\bf Proof} }{\endprf\par}

\title{\bf A Proximal Alternating Direction Method of Multiplier for Linearly Constrained Nonconvex Minimization\thanks{This research is supported in part by the NSFC grants 61731018 (key project) and 61571384, and by the Peacock project of Shenzhen Municipal Government. }}
\author{Jiawei Zhang\thanks{Shenzhen Research Institute of Big Data, The Chinese University of Hong Kong, Shenzhen, China.
	Email: luozq@cuhk.edu.cn}
\ and \
Zhi-Quan Luo$^\dagger$}

\begin{document}
\maketitle
\begin{abstract}
  Consider the minimization of a nonconvex differentiable function over a polyhedron. A popular primal-dual first-order method for this problem is to perform a gradient projection iteration for the augmented Lagrangian function and then update the dual multiplier vector using the constraint residual. However, numerical examples show that this approach can exhibit ``oscillation" and may not converge. In this paper, we propose a proximal alternating direction method of multipliers for the multi-block version of this problem. A distinctive feature of this method is the introduction of a ``smoothed" (i.e., exponentially weighted) sequence of primal iterates, and the inclusion, at each iteration, to the augmented Lagrangian function a quadratic proximal term centered at the current smoothed primal iterate. The resulting proximal augmented Lagrangian function is inexactly minimized (via a gradient projection step) at each iteration while the dual multiplier vector is updated using the residual of the linear constraints. When the primal and dual stepsizes are chosen sufficiently small, we show that suitable ``smoothing" can stabilize the ``oscillation", and the iterates of the new proximal ADMM algorithm converge to a stationary point under some mild regularity conditions.

The iteration complexity of our algorithm for finding an $\epsilon$-stationary solution is $\mathcal{O}(1/\epsilon^2)$, which improves the best known complexity of $\mathcal{O}(1/\epsilon^3)$ for the problem under consideration. Furthermore, when the objective function is quadratic, we establish the linear convergence of the algorithm. Our proof is based on a new potential function and a novel use of error bounds.
\end{abstract}


\section{Introduction}
Consider the following linearly constrained optimization problem:
\begin{equation}
\label{P}
\begin{array}{ll}
\mbox{minimize}& f(x)\\ [5pt]
\mbox{subject to} & Ax=b,\ x\in P,
\end{array}
\end{equation}
where $f$ is differentiable (not necessarily convex)and
\begin{equation}\label{eq:box}
P:=\left\{x\in\mathbb{R}^n\mid \ell_i\le x_i\le u_i,\ i=1,2,...,n \right\}
\end{equation}
is a bounded box with $\ell_i<u_i$ for all $i$, and the matrix $A$ has dimension $\mathbb{R}^{m\times n}$.
Problems of the form \eqref{P} arise in many applications involving big data, including nonnegative matrix factorization \cite{Zhang-10, Sun-Fevotte-14, Hajinezhad-16}, phase retrieval \cite{Wen-12}, distributed matrix factorization \cite{Ling-Xu-Yin-Wen-12}, polynomial optimization \cite{Jiang-Ma-Zhang-13}, asset allocation \cite{Wen-13}, zero variance discriminant analysis \cite{Ames-Hong-14}, to name just a few. A popular approach to solve problem \eqref{P} is to dualize the linear equality constraint and apply a primal-dual type algorithm to the resulting augmented Lagrangian function. Such approach is particularly attractive when the objective function $f(x)$ has a separable structure since in this case the corresponding primal minimization problem can be decomposed and often times solvable efficiently in parallel, while the dual update can be carried out in closed form.
The local convergence analysis of the classical augmented Lagrangian method was for smooth objective function with smooth equality constraint \cite{bertsekas}. Global convergence of a primal-dual method for smooth objective and smooth equality constraint was given recently in \cite{Hong17}. When the decision variable $x$ consists of many small variable blocks, the popular alternating direction method of multipliers (ADMM) is often the preferred algorithm to solve \eqref{P}, see \cite{boyd} for a detailed coverage of the method and many applications from a set of diverse fields. It is well known \cite{eckstein-bertsekas} that for the two-block (strongly) convex case ADMM can be viewed as a variant of proximal-point method or operator-splitting method, from which one can derive linear convergence of the method. The paper \cite{Ye} proves that the direct extension of ADMM to three-block case may not converges and state some sufficient conditions for the extended algorithm to converge.
The reference \cite{Luo-Hong12} uses the error bound analysis \cite{eb-survey} to establish the linear convergence rate of ADMM for a family of convex programming problem with any number of variable blocks with a reduced dual stepsize. However, for nonconvex problems, the convergence of the augmented Lagrangian method or ADMM has not been well understood, despite the fact they have been widely used in applications.
In \cite{Hong-Luo17}, the convergence of ADMM was established for some special nonconvex problems such as consensus-based sharing problems by using the augmented Lagrangian function as the potential function. This approach was further extended \cite{wyin16} to a larger family of nonconvex-nonsmooth problems under some technical assumption such as the prox-regularity of the objective function. The papers \cite{JiangMa16, GLI15, Loan18} proved the convergence of some inexact ADMM for certain nonsmoooth, nonconvex problems. However, these references all require at least one block of the variable to be unconstrained and a strong feasibility assumption holds, namely, suppose the linear equality constraint is $A_1x_1+A_2x_2=b$, then the image of $A_1$ is contained in the image of $A_2$ and the variable block $x_2$ does not have any other constraint.
Recently, \cite{Gao-Goldfarb-Curtis-2018} established the convergence of multi-block ADMM algorithm for the so called multi-affine constraints which are linear in each variable block but otherwise nonconvex. Similar to \cite{JiangMa16, GLI15, Loan18}, this paper also has some technical assumptions, including the similar strong feasibility constraint and that the objective function for some block must be strongly convex. Thus, the results from the existing studies \cite{JiangMa16, GLI15, Loan18, Gao-Goldfarb-Curtis-2018} do not cover the general problem \eqref{P}.
The penalty method of \cite{Monteiro} does apply to the linearly constrained nonconvex problem \eqref{P}. However, the penalty method is usually slower than the method of multipliers which do not use diminishing stepsizes or any unbounded parameters. To our knowledge, the fastest first order algorithm for the linearly constrained problem \eqref{P} is the one proposed by \cite{Monteiro}, which achieves an iteration complexity $\mathcal{O}(1/\epsilon^3)$.
The contribution of this paper is as follows. We propose a proximal alternating direction method of multipliers (ADMM) to solve the linearly constrained nonconvex differentiable minimization problem \eqref{P}. A distinctive feature of the algorithm is to introduce a ``smoothed" (i.e., exponentially weighted) sequence of primal iterates, and at each iteration add to the augmented Lagrangian function an extra quadratic proximal term centered at the smoothed primal iterate. The resulting proximal augmented Lagrangian function is inexactly minimized at each iteration while the dual multiplier vector is updated using the residual of the linear constraints. The algorithm is well suited for large scale optimization involving big data, and easily extends to the multiple variable block case, resulting in a variant of the well-known ADMM algorithm. When the primal and dual stepsizes are chosen sufficiently small, we show that the iterates of the proximal ADMM algorithm converge to a stationary point of the nonconvex problem under some mild regularity conditions. Furthermore, we show that this algorithm has an iteration complexity $\mathcal{O}(1/\epsilon^2)$ to find an $\epsilon$-stationary solution of problem \eqref{P}, which improves the best known complexity of $\mathcal{O}(1/\epsilon^3)$ \cite{Monteiro}. Moreover, we present a numerical example showing that the ``smoothing" step is necessary for the convergence of the proximal ADMM when the objective function is nonconvex. Furthermore, for a quadratic objective function, we establish the linear convergence of the algorithm.
\section{Preliminaries}
\subsection{Some notations in this paper}
We define some notations used in this paper.
\begin{itemize}
\item $\mathbb{R}^n$ is the $n-$dimension Euclidian space.
\item $\mathbb{R}^{m\times n}$ is the space of all $m\times n$ real matrix.
\item$[\cdot]_+$ that  satisfies $[x]_+=\arg\min_{\bar{x}\in P}\|x-\bar{x}\|$ is the  projection operator to the box $P$. In particular, the $i-$th coordinate of $[x]_+$ is given by
\begin{equation}
([x]_+)_i=\left\{\begin{array}{ll}
\ell_i,&\mbox{if}\ x_i<\ell_i,\\
x_i,&\mbox{if}\ x_i \in [\ell_i, u_i],\\
u_i,&\mbox{if}\ x_i>u_i.
\end{array}
\right.
\end{equation}
\item $(\cdots)_+$ is the projection operator to the non-negative orthant.
\item $\succeq$ means the coordinate-wise $\ge$ and $\preceq$ means the coordinate-wise $\le$.
\end{itemize}

\subsection{The set of stationary solutions}
We first define the solution of the problem \eqref{P} in this subsection.
Due to the linearity of the constraints, there exists a set of Lagrangian multipliers for each stationary point of \eqref{P} such that the KKT condition holds. We denote the set of stationary points of \eqref{P} by $X^*$.
Writing down the KKT condition, letting $y, \mu, \nu$ be the multipliers, we have:
\begin{eqnarray}
\nabla f(x^*)+A^Ty^*-\mu^*+\nu^*&=&0,\nonumber\\
Ax^*&=&b,\nonumber\\
\ell_i\le\ \; x_i^*&\le& u_i,\quad \mbox{for all }i\nonumber\\
\mu^*&\succeq& 0,\nonumber\\
\nu^*&\succeq& 0,\nonumber\\
\mu_i^*(\ell_i-x_i^*)&=&0, \quad \mbox{for all }i\label{comple1}\\
\nu_i^*(x_i^*-u_i)&=&0,\quad \mbox{for all }i\label{comple2},
\end{eqnarray}
where $\mu_i^*$ and $\nu_i^*$ denote the $i$-th component of $\mu^*$ and $\nu^*$ respectively.
Let $X^*, Y^*$ be the sets of all $x^*$ and $y^*$ satisfying the KKT condition. And let $W=\{(x, y)\mid x, y\ is\ a\ pair\ of\ primal-dual\ solution \}$.
Note that \eqref{comple1} and \eqref{comple2} are the complementarity conditions. It means that either $(\ell_i-x_i^*)$ or $\mu_i^*$ must be zero for all $i$ and similarly for $\nu^*$.
A stronger condition, which holds generically, is called ``strict complementarity condition $\Gamma$''.
\begin{definition}
If for all solutions $(x^*, y^*, \mu^*, \nu^*)$ of the KKT system, for any $i$, exactly one of $\mu_i^*$ and $(\ell_i-x_i^*)$ is zero and exact one of $\nu_i^*$ and $(x_i^*-u_i)$ is zero, then we say the original problem satisfies the strict complementarity condition.
\end{definition}
\subsection{Assumptions}
In this subsection, we give our main assumptions, which are valid in many practical problems.
\begin{assumption}
\label{feasible-int}
$ $ \\ [-10pt]
\begin{enumerate}
\item [{\rm (a)}] The origin is in the relative interior of the set $AP-b=\{Ax-b\mid x\in P\}$.
\item [{\rm (b)}] The strict complementarity condition holds for \eqref{P}.
\item [{\rm (c)}] The objective function $f$ is a differentiable function with Lipschitz continuous gradient
\[
\|\nabla f(x)-\nabla f(x')\|\le L\|x-x'\|,\quad \mbox{ for some }L>0 \mbox{ and } \forall\; x, x'\in P.
\]
\end{enumerate}
\end{assumption}
Note that Assumption~\ref{feasible-int}(a) is equivalent to the feasibility of \eqref{P} for all slightly perturbed $b$ from the range space of $A$; in particular it does not require the full row rank of $A$. By using Cauchy-schwarz inequality, assumption~\ref{feasible-int}(c) implies the existence of a constant $\gamma$ (possibly negative) such that
\begin{equation}\label{gamma}
\langle \nabla f(x)-\nabla f(x'),x-x'\rangle \ge \gamma \|x-x'\|^2, \quad \mbox{for all $x, x'\in P$}.
\end{equation}
Assumption~\ref{feasible-int}(b) is reasonable since the strict complementarity is valid generically, as we argue in the proposition below.
\begin{proposition}
Suppose $f(x)=g(x)+v^Tx$, $g$ is Lipschitz-differentiable and $v$ is a constant vector. If the data vector $(v, b)$ is generated from a continuous probability distribution, then with probability $1$, the strict complementarity condition holds for \eqref{P}. Here a probability distribution is said to be continuous if the probability of a Lesbergue-zero-measure set is zero.
\end{proposition}
\begin{proof}
We will use the fact that Lipschitz continuous functions map a zero-measure set to a zero-measure set.
For active sets $S_1$ and $S_2$ the KKT condition with respect to $\{1,2,...,n\}$ is
\begin{eqnarray*}
\nabla g(x)+A^Ty-\mu+\nu&=&-v\nonumber\\
Ax&=&b,\nonumber\\
x_i&=&\ell_i, \ \mu_i\ge 0, \quad i\in S_1\\
x_i&=&u_i, \ \nu_i\ge 0, \quad i\in S_2\\
x_i&\in&[\ell_i, u_i],\ \mbox{ for all }i\nonumber\\
\mu_i&=&0, \quad i\notin S_1\\
\nu_i&=&0,\quad i\notin S_2.
\end{eqnarray*}
We prove that for any $S_1, S_2\subseteq \{1,2,...,n\}$, with probability $0$, the strictly complementary condition does not hold.
Since $S_1, S_2$ have only finitely many choices, we only need to consider fixed $S_1$ and $S_2$.
Without loss of generality, assume that $x_1=\ell_1$, $\mu_1=0$ and $1\notin S_1$.
Consider the Lipschitz continuous map $\Phi$
$$\Phi(x, y, \mu, \nu)=(\nabla g(x)+A^Ty-\mu+\nu, Ax)$$
from the set
\begin{eqnarray*}
T=\left\{(x ,y, \mu, \nu)\right.&\mid& x_i=\ell_i, i\in S_1, \ x_i=u_i, i\in S_2, \\
&&\ell_i\le x_i\le u_i,\ \mu_i=0,\ i \notin S_1,\ \nu_i=0,\ i\notin S_2, \\
&&\left. \mu\succeq 0,\ \nu\succeq 0, x_1=\ell_1\right\}
\end{eqnarray*}
to $\mathbb{R}^{n+m}$.
Clearly $\Phi$ maps from a $n+m-1$ dimension subset to $n+m$-dimension space.
Hence, the image is zero-measure in $\mathbb{R}^{n+m}$.
Consequently, the choice for $(v, b)$ such that the solution exists is of measure zero. Hence, the probability for it is $0$ since the probability distribution is continuous.
\end{proof}
\subsection{A Proximal Inexact Augmented Lagrangian Multiplier Method}
We will state our algorithm in this subsection based on the augmented Lagrangian function.
The Augmented Lagrangian function for \eqref{P} is given by
$$L(x; y)=f(x)+y^T(Ax-b)+\frac{\Gamma}{2}\|Ax-b\|^2,$$
where $\Gamma>0$ is a constant.
The classical augmented Lagrangian multiplier method minimizes $L(x;y)$ for a fixed $y$ over the box constraint $P$, and then updates $y$ using the residual of the primal equality constraint $Ax=b$. Unfortunately, due to the nonconvexity of $f$, the exact minimization of $L(x;y)$ with respect to $x$ can be difficult. Thus, it is often more practical to minimize $L(x;y)$ inexactly with respect to $x$. In particular, we recall the following simple inexact augmented Lagrangian multiplier method (which also corresponds to the linearized ADMM algorithm when there is only one primal variable block). 
\begin{algorithm}[ht]
\caption{An Inexact Augmented Lagrangian Multiplier Method}
\label{Alg}
\begin{algorithmic}[1]
\STATE Let $\alpha>0$ and $c>0$;
\STATE Initialize $x^0,\ y^0$;
\FOR{$t=0,1,2,\ldots,$}
\STATE $y^{t+1}=y^t+\alpha(Ax^t-b)$;
\STATE $x^{t+1}=[x^t-c\nabla_x L(x^t; y^{t+1})]_+$.
\ENDFOR
\end{algorithmic}
\end{algorithm}
Though easy to implement numerically, the above inexact augmented Lagrangian multiplier method can behave erratically or even diverge for nonconvex problems (see Figure 1 in Section 6 for a numerical example). To stabilize the convergence behaviour of the inexact augmented Lagrangian multiplier method, we propose a proximal version of the augmented Lagrangian multiplier method. In this new method, we introduce an exponential averaging (or smoothing) scheme to generate an extra sequence $\{z^t\}$ and insert an extra quadratic proximal term centered at $z^t$ to the augmented Lagrangian function so that the next primal iterate $x^{t+1}$ does not deviate too much from the stabilized iterate $z^t$. More specifically,
let
\begin{equation}\label{eq:K}
K(x, z;y)=L(x; y)+\frac{p}{2}\|x-z\|^2,
\end{equation}
where $p$ is a positive parameter. Note that the function $K$ is Lipschitz differentiable with modulus $L_K=L+p+\Gamma\sigma^2$, and can be made strongly convex in $x$ with modulus $\gamma_K=p+\gamma>0$ if $p$ is chosen to be larger than $-\gamma$.
We consider the following proximal inexact augmented Lagrangian multiplier method.
\begin{algorithm}[ht]
\caption{A Proximal Inexact Augmented Lagrangian Multiplier Method}
\label{Alg2}
\begin{algorithmic}[1]
\STATE Let $\Gamma>0$, $\alpha>0$, $0<\beta\le 1$ and $\frac{1}{L_K}>c>0$;
\STATE Initialize $x^0\in P,\ z^0\in P,\ y^0\in \mathbb{R}^m$;
\FOR{$t=0,1,2,\ldots,$}
\STATE $y^{t+1}=y^t+\alpha(Ax^t-b)$;
\STATE $x^{t+1}=[x^t-c\nabla_x K(x^t, z^t; y^{t+1})]_+$;
\STATE $z^{t+1}=z^t+\beta(x^{t+1}-z^t)$.
\ENDFOR
\end{algorithmic}
\end{algorithm}
Our main claim is that the introduction of the proximal term can ensure the global convergence of Algorithm~\ref{Alg2} for the nonconvex problem \eqref{P}.
\begin{theorem}\label{th:main}
Suppose Assumption \ref{feasible-int} holds. Moreover, suppose the parameters $c$ and $p$ are selected to satisfing $\frac{1}{L_K}>c>0, p>-\gamma$ and that the primal and dual stepsizes $\beta$ and $\alpha$ to be sufficiently small. Then the dual iterates $\{y^t\}$ are bounded. Moreover, there holds
\begin{eqnarray}
\lim_{t\to \infty}\|x^{t+1}-x^t\|&=&0,\nonumber\\
\color{black}\lim_{t\to \infty}\dist(x^t, X^*)&=&0,\nonumber\\
\lim_{t \to \infty}\dist(z^t, X^*)&=&0,\nonumber\\
\lim_{t\rightarrow \infty}\dist((x^t, y^t), W)&=&0\nonumber
\end{eqnarray}
and every limit point of the sequence $\{(x^t, y^t)\}$ generated by Algorithm \ref{Alg2} is a primal-dual stationary point of \eqref{P}.
\end{theorem}
Let
\begin{eqnarray}
d(y, z)&=&\min_{x\in P}K(x, z;y),\label{d}\\
x(y, z)&=&\argmin_{x\in P}K(x, z;y),\label{K}\\
M(z)&=&\min_{x\in P, Ax=b}(f(x)+\frac{p}{2}\|x-z\|^2),\label{proximal}\\
x^*(z)&=&\argmin_{x\in P,\ Ax=b}\left(f(x)+\frac{p}{2}\|x-z\|^2\right).\label{x}
\end{eqnarray}
It should be noted that if $p>-\gamma$, then $f(x)+\frac{p}{2}\|x-z\|^2$ is strongly convex, so there holds
\begin{equation}\label{bounds}
K(x,z;y)\ge d(y,z),\quad M(z)\ge d(y,z),\quad \forall \ y,z,
\end{equation}
where the first inequality is due to \eqref{d}, while the second inequality follows from the strong duality
\[
M(z)=\max_yd(y,z).
\]

Note that the subproblem of the Algorithm \ref{Alg2} is easy since it involves only projection to the box $P$.
Compared to Algorithm \ref{Alg}, the proximal inexact augmented Lagrangian method (Algorithm~\ref{Alg2}) constructs an auxiliary sequence $\{z^t\}$ which is a recursive average of the primal sequence $\{x^t\}$, and uses it to build a quadratic proximal term in the augmented Lagrangian function $L(x^t;y^t)$. Notice that the recursive averaging step is computationally simple, so Algorithm~\ref{Alg2} has a similar per-computational complexity to Algorithm~\ref{Alg}. It should be noted that this new quadratic proximal term in the augmented Lagrangian function introduces an extra term $p(x^t-z^t)$ in the gradient of $L(x^t;y^{t+1})$. This extra term is an exponentially weighted average of {\it all} the previously generated primal iterates $\{x^0,x^1,...,x^t\}$. As such, it is different from the well known ``momentum term" in the backpropagation training algorithm which is equal to the difference of the previous two primal iterates. Also, this extra term is different from the Nesterov's acceleration scheme which adds to the gradient descent direction a specific (iteration dependent) average of previous two primal iterates.
In the rest of the paper, we fix parameters $c<1/L_K$ and $p>-\gamma$.
\section{Convergence Analysis}
We will prove Theorem \ref{th:main} in this section. The proof consists of a series of lemmas. We will give the proof of the key lemmas and prove the rest in Appendix~\ref{Appendix:B}, \ref{Appendix:C} and \ref{Appendix:D}.
\subsection{Key Lemmas}
In this subsection, we will give some key lemmas that are needed to establish the main theorem.
\subsubsection{Three Descent Lemmas}\label{sub:3descent}
In this subsection we give three descent lemmas to estimate the changes in the primal, the dual and the proximal function respectively after one iteration of Algorithm~\ref{Alg2}.
\begin{lemma}[Primal Descent] \label{primal}
For any t, if $c<1/L_K,$
$$K(x^t, z^t; y^t)-K(x^{t+1}, z^{t+1}; y^{t+1})\ge \frac{1}{2c}\|x^t-x^{t+1}\|^2+\frac{p}{2\beta}\|z^t-z^{t+1}\|^2-\alpha\|Ax^t-b\|^2.$$
\end{lemma}
\begin{proof}
First, we have the trivial equality:
\begin{equation}\nonumber
K(x^t, z^t; y^t)-K(x^t, z^t; y^{t+1})
=-\alpha\|Ax^t-b\|^2.
\end{equation}
Next, notice that updating $x$ is a standard gradient projection, hence,
\begin{equation}\label{eq:descent}
K(x^t, z^t; y^{t+1})-K(x^{t+1}, z^t; y^{t+1})
\ge \frac{1}{2c}\|x^t-x^{t+1}\|^2.
\end{equation}
Moreover, recall that in Algorithm~\ref{Alg2}, $z^{t+1}=z^t+\beta(x^{t+1}-z^t)$, we have
\begin{eqnarray}
K(x^{t+1}, z^{t}; y^{t+1})-K(x^{t+1}, z^{t+1}; y^{t+1}) &=& \frac{p}{2}(\|x^{t+1}-z^t\|^2-\|x^{t+1}-z^{t+1}\|^2) \nonumber\\
&=& \frac{p}{2}(z^{t+1}-z^{t})^T((x^{t+1}-z^t)+(x^{t+1}-z^{t+1})) \nonumber\\
&=& \frac{p}{2}(2/\beta-1)\|z^t-z^{t+1}\|^2 \nonumber\\
&\ge&\frac{p}{2\beta}\|z^t-z^{t+1}\|^2,\label{zprimal}
\end{eqnarray}
for $\beta\le 1$.
Combining the above three inequalities yields the desired result.
\end{proof}
\begin{lemma}[Dual Ascent]\label{dual ascent}
For any $t$, we have
\begin{eqnarray*}
d(y^{t+1}, z^{t+1})-d(y^t, z^t)
\ge && \alpha(Ax^t-b)^T(Ax(y^{t+1}, z^t)-b)\\
&&+\frac{p}{2}(z^{t+1}-z^t)^T(z^{t+1}+z^t-2x(y^{t+1}, z^{t+1})).
\end{eqnarray*}
\end{lemma}
\begin{proof}
First, recall that
$$K(x, z;y)=f(x)+\langle y, Ax-b\rangle+\frac{\Gamma}{2}\|Ax-b\|^2+\frac{p}{2}\|x-z\|^2$$
so we have
\begin{eqnarray*}
d(y^{t+1}, z^t)-d(y^{t}, z^t)
&=& K(x(y^{t+1}, z^t), z^t; y^{t+1})-K(x(y^{t}, z^t), z^t; y^t)\nonumber\\
&\ge &K(x(y^{t+1}, z^t), z^t; y^{t+1})-K(x(y^{t+1}, z^t), z^t; y^t)\nonumber\\
&=& \langle y^{t+1}-y^t,Ax(y^{t+1}, z^t)-b\rangle,\nonumber\\
&=& \alpha(Ax^t-b)^T(Ax(y^{t+1}, z^t)-b), \nonumber
\end{eqnarray*}
where the inequality is because $x(y^t, z^t)=\arg\min_xK(x, z^t; y^t)$.
Next, using the same technique, we have
\begin{eqnarray}
d(y^{t+1},z^{t+1})-d(y^{t+1}, z^{t})&=& K(x(y^{t+1}, z^{t+1}), z^{t+1}; y^{t+1})-K(x(y^{t+1}, z^t), z^t; y^{t+1})\nonumber\\
&\ge& K(x(y^{t+1}, z^{t+1}), z^{t+1}; y^{t+1})-K(x(y^{t+1}, z^{t+1}), z^t; y^{t+1})\nonumber\\
&=& \frac{p}{2}(\|x(y^{t+1}, z^{t+1}-z^{t+1}\|^2-\|x(y^{t+1}, z^{t+1})-z^t\|^2)\nonumber\\
&=&\frac{p}{2}(z^{t+1}-z^t)^T(z^{t+1}+z^t-2x(y^{t+1}, z^{t+1})). \label{zdual}
\end{eqnarray}
Combining these, we get the desired result.
\end{proof}
\begin{lemma}[Proximal Descent]\label{proximal-descent}
For any $t\ge 0$, there holds
\begin{equation}\label{eq:prox-descent}
M(z^{t+1})-M(z^{t})\le p(z^{t+1}-z^t)^T(z^t-x^*(z^t))+\frac{p\tilde{L}}{2}\|z^t-z^{t+1}\|^2,
\end{equation}
where $\tilde{L}=\frac{p}{p+\gamma}+1$.
\end{lemma}
The proof of Lemma~\ref{proximal-descent} will be given later since it needs the error bound in Lemma~\ref{error bound}.

The three terms $K(x^t, z^t; y^t)$, $-d(y^t, z^t)$ and $M(z^t)$ individually do not need to decrease after each iteration; however, some weighted sum of them does! This is the main idea of the proof of Theorem~\ref{th:main}.
To establish the convergence of Algorithm~\ref{Alg2}, we construct a potential function which decreases sufficiently after each iteration. This potential function is a linear combination of the primal, dual and proximal terms .
Specifically, we will prove that the potential function
$$\phi^t=\phi(x^t, z^t;y^t)=K(x^t, z^t; y^t)-2d(y^t, z^t)+2M(z^t)$$
decreases sufficiently after each iteration for sufficiently small $\alpha$ and $\beta$, where the functions $K$, $d$ and $P$ are defined in \eqref{eq:K}-\eqref{proximal}.
Let $\underline{f}=\min_{x\in P, Ax=b}f(x)$. Then $M(z)\ge \underline{f}$ for any $z\in P$.
Note that $x^t\in P$ and $z^t\in P$ for all $t$ (see the definition of the Algorithm \ref{Alg2}), so $M(z^t)$ is bounded from below. Moreover, it follows from \eqref{bounds} that
\begin{eqnarray}\label{phi}
\phi^t=(K(x^t, z^t; y^t)-d(y^t, z^t))+(M(z^t)-d(y^t, z^t))+M(z^t)\ge M(z^t)\ge \underline{f}
\end{eqnarray}
is also bounded below.
When considering $\phi^t-\phi^{t+1}$, we have some negative terms, which need to be bounded. To this aim, we need to some so-called ""error bounds'', which bound the iteration sequences by the residuals.
\subsubsection{Error Bounds}
In this subsection, we prove some error bounds, which bound the distance of the iterative points to the solution set and bound the difference of $x(\cdot, \cdot), x^*(\cdot)$ when $y$ and $z$ are perturbed.
The following lemma implies that if the dual residual $Ax-b\to 0$, the dual variable $y$ is bounded.
\begin{lemma}\label{ybounded}
Suppose Assumption \ref{feasible-int}$\;($a$)$ holds. If $\{y^i\}\subseteq \mathbb{R}^m$ be a sequence such that $y^i\in \tilde{y}^0+\mathrm{range}(A)$, where $\tilde{y}^0$ is some fixed vector . If $\|Ax(y^i, z^i)-b\|\rightarrow 0$ for some sequence $\{z^i\}$,
then $\{y^i\}$ is bounded.
\end{lemma}
\begin{proof}
According to Assumption \ref{feasible-int}(a), there exists a positive $r>0$, such that for any direction $d\in Range(A)$, we can find an $x\in P$ satisfying $\|Ax-b\|=r$ and $Ax-b$ has the same direction as $d$.
We claim that if $\|y^i\|$ goes to infinity, then $\|Ax(y^i, z^i)-b\|$ must be bounded away from $0$, i.e., we can not have $\|Ax(y^i, z^i)-b\|\rightarrow 0$.
We prove this by contradiction. Assume that $\|y^i\|\rightarrow \infty$ and $\|Ax(y^i, z^i)-b\|\rightarrow 0$ for some sequences $\{y^i\}$ and $\{z^i\}$. Since $\tilde{y}^0$ is fixed, it follows that $\|\tilde{y}^i\|\to \infty$. By Assumption \ref{feasible-int}(a), there exists a $x^i\in P$ such that $Ax^i-b$ is of the same direction as $-\tilde{y}^i$ and $\|Ax^i-b\|=r$.
Let
$$M=\max_{x, z\in P}\left\{|f(x)|+\frac{p}{2}\|x-z\|^2+\frac{\Gamma}{2}\|Ax-b\|^2+\langle \tilde{y}^0, Ax-b\rangle\right\}.$$
When $i$ is sufficiently large, we have $\|\tilde{y}^i\|>4M/r$ and $\|Ax(y^i,z^i)-b\|<r/2$. Therefore, it follows that
\begin{equation}\label{eq2-1}
(\tilde{y}^i)^T(Ax^i-b)=-\|\tilde{y}^i\|\cdot \|Ax^i-b\|=-\|\tilde{y}^i\|r<-4M,
\end{equation}
\begin{eqnarray*}
(\tilde{y}^i)^T(Ax^i-b)&=& -r\|\tilde{y}^i\|\\
&\le& -r \frac{|(\tilde{y}^i)^T(Ax(y^i, z^i)-b)|}{\|Ax(y^i, z^i)-b\|}\\
&\le & -r\frac{|(\tilde{y}^i)^T(Ax(y^i, z^i)-b)|}{r/2}\\
&\le&2(\tilde{y}^i)^T(Ax(y^i, z^i)-b),
\end{eqnarray*}
where the second step follows from Cauchy-Schwartz inequality.
Hence, we have
\begin{eqnarray*}
K(x^i, z^i; y^i)-K(x(y^i, z^i), z^i; y^i)&\le& 2M+(\tilde{y}^i)^T(Ax^i-b)-(\tilde{y}^i)^T(Ax(y^i, z^i)-b)\\
&\le& 2M+(\tilde{y}^i)^T(Ax^i-b)-\frac{1}{2}(\tilde{y}^i)^T(Ax^i-b)\\
&=&2M+\frac{1}{2}(\tilde{y}^i)^T(Ax^i-b)<0,
\end{eqnarray*}
where the last step is due to \eqref{eq2-1}.
This further implies
$$K(x^i, z^i; y^i)<K(x(y^i, z^i), z^i; y^i)$$
which contradicts the definition of $x(y^i, z^i)$.
Hence, when $\|Ax(y^i, z^i)-b\|\rightarrow 0$, we have that $\|y^i\|$ is bounded.
\end{proof}
\noindent{\bf Remark} Note that each iterate $y^t$ of Algorithm \ref{Alg2} satisfies the condition of the above lemma with $\tilde{y}^0=y^0$.

The boundedness of $y^t$ will be used to establish the dual error bound later in Lemma~\ref{error bound}.
Next result shows that $x^*(z)$ is continuous in $z$ and $x(y, z)$ is Lipschitz-continuous in $(y, z)$.
\begin{lemma}\label{continuity}
Suppose $p>-\gamma$ and $z, z'\in P$, then we have
\begin{eqnarray*}
&&\|x^*(z)-x^*(z')\|\le\frac{p}{p+\gamma}\|z-z'\|;\\
&&\|x(y, z)-x(y, z')\|\le \frac{p}{p+\gamma}\|z-z'\|;\\
&&\|x(y, z)-x(y', z)\|\le \frac{p+\gamma}{\sigma}\|y-y'\|.
\end{eqnarray*}
\end{lemma}
\begin{proof}
We only prove the Lipschitz continuity of $x^*(\cdot)$, the other claims can be proved similarly. Let $f(x;z)=f(x)+\frac{p}{2}\|x-z\|^2$.
\begin{eqnarray*}
&&f(x^*(z); z')-f(x^*(z'); z'))\\
&=&(f(x^*(z); z)-f(x^*(z'); z))-(f(x^*(z'); z')-f(x^*(z'); z))+(f(x^*(z); z')-f(x^*(z); z))\\
&=&(f(x^*(z); z)-f(x^*(z'); z))-\frac{p}{2}(-2(z'-z)^Tx^*(z')+\|z'\|^2-\|z\|^2)\\
&&+\frac{p}{2}(-2(z'-z)^Tx^*(z)+\|z'\|^2-\|z\|^2)\\
&=&(f(x^*(z); z)-f(x^*(z'); z))+{p}(z'-z)^T(x^*(z')-x^*(z))\\
&\le &-\frac{p+\gamma}{2}\|x^*(z)-x^*(z')\|^2+{p}(z'-z)^T(x^*(z')-x^*(z)),
\end{eqnarray*}
where the last inequality is due to the strong convexity of $f(x; z)$ in variable $x$.
On the other hand, again by the strong convexity, we have
$$f(x^*(z); z')-f(x^*(z'); z')\ge \frac{p+\gamma}{2}\|x^*(z)-x^*(z')\|^2.$$
Hence, we have
$$-(p+\gamma)\|x^*(z)-x^*(z')\|^2+{p}(z'-z)^T(x^*(z')-x^*(z))\ge 0,$$
which by Cauchy-Schwartz inequality further implies
$$\|x^*(z)-x^*(z')\|\le\frac{p}{p+\gamma}\|z-z'\|.$$
Hence, $x^*(\cdot)$ is Lipschitz continuous with modulus $p/(p+\gamma)$.\endprf
\end{proof}

We denote $x^+, y^+, z^+$ be the updated variables of $x, y, z$ by Algorithm \ref{Alg2}, namely,
\begin{eqnarray}\label{plus}
y^+&=&y+\alpha(Ax-b);\\
x^+&=&[x-c\nabla K(x, z; y^+)]_+;\\ 
z^+&=&z+\beta(x^+-z).
\end{eqnarray}
The following simple lemma says that if the algorithm stops, it finds a pair of primal-dual solution.
\begin{lemma}\label{limitpoint}
If
$$(x, y, z)=(x^+, y^+, z^+),$$
Then $(x, y)\in W$ is a pair of solution.
\end{lemma}
\begin{proof}
The proof is just to check the KKT conditions and is omitted since it is trivial.
\end{proof}
A couple of corollaries are in order, which are direct from the two lemmas above.
\begin{corollary}\label{coro1}
Suppose that Assumption~\ref{feasible-int}(a) holds. For any $\epsilon>0$, there exists a $\delta(\epsilon)>0$, such that for vectors $x,\; z\in P$ and any $y\in \tilde{y}^0+\mathrm{range}(A)$ with a fixed $\tilde{y}^0$ satisfying
$$\max\{\|x-x^+\|, \|z-x^+\|, \|Ax(y^+, z)-b\|\}<\delta(\epsilon),$$
we have
$$
\max\{\dist(x,X^*), \dist(z, X^*)\}<\epsilon.
$$
\end{corollary}
The proof of Corollary~\ref{coro1} is relegated to Appendix~\ref{Appendix:B}.
\begin{corollary}\label{coro2}
Suppose that Assumption~\ref{feasible-int}(a) holds. For any $\epsilon>0$, there exists a $\delta(\epsilon)>0$ such that if
$$\|Ax(y^+, z)-b\|<\delta(\epsilon)$$
for some $y\in \tilde{y}^0+\mathrm{range}(A)$ with some fixed $\tilde{y}^0$,
then
$$\|x(y^+, z)-x^*(z)\|< \epsilon,\mbox{ for any $z\in P$}.$$
\end{corollary}
The proof of Corollary~\ref{coro2} can be seen in Appendix~\ref{Appendix:C}.

Next we develop some primal and dual error bounds. To prove the dual error bound, we need to make use of the Hoffman bound, which can be seen in \cite{Pang-complementarity}.
\begin{proposition}\label{hoffman}
Let $A\in \mathbb{R}^{m\times n},\ C\in \mathbb{R}^{k\times n}$ and $b\in \mathbb{R}^{m} ,d\in \mathbb{R}^{k}$, then the distance from a point $\bar{x}\in \mathbb{R}^n$ to the set $S=\{x\mid Ax\le b, Cx=d\}$
is bounded by:
$$\dist(\bar{x}, S)^2\le\theta^2(\|(A\bar{x}-b)_+\|^2+\|C\bar{x}-d\|^2),$$
where $(\cdot)_+$ means the projection to the nonegative orthant and $\theta$ is a positive constant depending on $A$ and $C$ only.
\end{proposition}

\begin{lemma}[Error Bounds]\label{error bound}
Suppose $p>-\gamma$, $\Gamma>0$ are fixed. Then there exist positive constants $\sigma_1,...,\sigma_4>0$ $($independent of $y^{t+1}$ and $z^t$$)$ such that the following error bounds hold:
\begin{eqnarray}
\|x^{t+1}-x^t\|&\ge& \sigma_1\|x^t-x(y^{t+1}, z^t)\|,\label{eb1}\\
\|x^{t+1}-x^t\|&\ge& \sigma_2\|x^{t+1}-x(y^{t+1}, z^t)\|, \label{eb2} \\
\|y-y'\|&\ge&\sigma_3\|x(y,z)-x(y',z)\|, \label{eb6}\\
\|z^t-z^{t+1}\| &\ge &\sigma_4 \|x^*(z^t)-x^*(z^{t+1})\|,\label{eb4} \\
\|z^t-z^{t+1}\|&\ge & \sigma_4\|x(y^{t+1},z^t)-x(y^{t+1}, z^{t+1})\|,\label{eb3}
\end{eqnarray}
where $\sigma_1=c\gamma_K=c(p+\gamma)$, $\sigma_2={\sigma_1}/({1+\sigma_1})$, $\sigma_3=\gamma_K/\sigma=(\gamma+p)/\sigma$ and  ${\sigma}_4=\gamma_K/p=(\gamma+p)/p$.
Furthermore, suppose Assumption~\ref{feasible-int}$(${b}$)$ holds, then there exist positive scalars $\Delta,\; \sigma_5$ such that
\begin{eqnarray}
\dist(y, Y^*(z))&\le &\sigma_5\|Ax(y, z)-b\|, \label{eb5}
\end{eqnarray}
if $y\in y^0+\mathrm{range}(A)$, $\|Ax(y, z)-b\|\le \Delta$ and $\dist(z, X^*)\le \Delta$, where $Y^*(z)$ denotes the solution set of dual multipliers for \eqref{proximal}. 
\end{lemma}
\begin{proof}
The proof of inequalities \eqref{eb1} and \eqref{eb2}are standard, we prove them in Appendix~\ref{Appendix:D}. Also note that \eqref{eb6}, \eqref{eb4} and \eqref{eb3} are just Lemma \ref{continuity}.
It remains to prove the dual error bound \eqref{eb5}. To this end, we first write down, for any $z$, the optimality conditions for the strongly convex proximal optimization problem \eqref{proximal} (note that $\gamma_K=p+\gamma>0$)
as follows:
\begin{eqnarray}
\nabla f(x^*(z))+p(x^*(z)-z)+A^Ty^*(z)-\mu(z)+\nu(z) &=& 0, \nonumber\label{station}\\
Ax^*(z) &=& b, \nonumber \\
\mu_i(z)(x_i-\ell_i)&=&0, \quad i=1,\cdots, n,\label{kktforproximal}\\
\nu_i(z)(x_i-u_i)&=&0, \quad i=1,\cdots, n, \nonumber\\
\mu_i(z),\; \nu_i(z)&\ge& 0,\quad i=1, \cdots, n, \nonumber
\end{eqnarray}
where $y^*(z),\; \mu(z),\; \nu(z)$ are the Lagrangian multipliers.
For any $z$, note that $y^*(z)$ only appears in the term $A^Ty^*(z)$. By replacing $y^*(z)$ with its projection to $\mathrm{range}(A)$ if necessary, we can assume, without loss of generality, $y^*(z)\in \mathrm{range}(A)$.
If $z\in X^*$, by strong convexity, the unique optimal solution of \eqref{proximal} is given by $x^*(z)=z$.
Moreover, for $z\in X^*$, the strict complementarity condition
\begin{eqnarray*}
\mu_i(z)=0\quad&\Longrightarrow& \quad x_i^*(z)=z_i>\ell_i,\ i=1, \cdots, n, \\
\nu_i(z)=0\quad & \Longrightarrow&\quad x_i^*(z)=z_i<u_i,\ \ i=1, \cdots, n,
\end{eqnarray*} holds.
Now for any $y, z$, let $I(y,z),\; I(z)$ denote the set of inactive inequality constraints (the inequality constraint holds strictly) in problem \eqref{K} at $x(y,z)$ and \eqref{proximal} at $x^*(z)$ respectively. We prove that there exists a $\Delta>0$, such that
$$
I(y, z)=I(z),\mbox{ if $\|Ax(y, z)-b\|\le \Delta$ and $\dist(z,X^*)\le \Delta$.}
$$
We prove it by contradiction. Suppose the contrary, then there is a sequence $\Delta_k\to 0$ and sequences $\{y_k\}\subseteq y^0+\mathrm{range}(A), \{z_k\}$ with
$$\|Ax(y_k, z_k)-b\|\le \Delta_k,\quad \dist(z_k,X^*)\le \Delta_k,$$
such that $I(y_k, z_k)\ne I(z_k)$, for all $k$.
Note that $y^*(z_k)$ is the optimal dual solution for the problem \eqref{proximal} with $z=z_k$, so we have $x(y^*(z_k), z_k)=x^*(z_k)$ and $\|Ax(y^*(z_k), z_k)-b\|=0$, for all $k$. By Lemma~\ref{ybounded}, we know that $\{y_k\}, \{y^*(z_k)\}$ are bounded.
So we can assume that (passing to a subsequence if necessary)
$$y_k\rightarrow y^*,\ \mbox{ for some }y^*\in \mathbb{R}^m,$$
$$y^*(z_k)\rightarrow \bar{y},\ \mbox{ for some }\bar{y}\in \mathbb{R}^m,$$
$$z_k\rightarrow z^*,\ \mbox{ for some }z^*\in X^*.$$
According to Lemma \ref{continuity}, we have $x^*(z^k)\rightarrow x^*(z^*)$ and $x(y_k, z_k)\rightarrow x(y^*, z^*)$.
We also have $Ax(y^*, z^*)-b=0$, hence $x(y^*, z^*)=x(z^*)$.
Since for any $i$ at most one of $\mu_i$ and $\nu_i$ can be nonzero, $\{\mu_i(z_k)\}_{k=1}^{\infty}, \{\nu_i(z^k)\}_{k=1}^{\infty}$ are bounded by \eqref{station} and the fact that $y^{*}(z_k)$ is bounded. Hence, passing to a subsequence if necessary, we can assume that there exist $\mu^*$, $\nu^*$, such that
$$\mu_i(z_k)\rightarrow \mu^*, \ \ \nu_i(z_k)\rightarrow \nu^*,\ \mbox{ for all }i.$$
Hence, $(x^*(z^*), \bar{y}, \mu^*, \nu^*)$ is a solution to the KKT system \eqref{kktforproximal}. Using the strict complementarity property at $z^*\in X^*$, we have $\mu_i, \nu_i\ne 0, i\notin I(z^*)$. When $k\rightarrow \infty$, we have
\begin{eqnarray*}
x_i^*(z_k)\in (\ell_i, u_i), &&i\in I(z^*),\\
\mu_i>0, &&i\notin I(z^*)\mbox{ and } x_i^*(z_k)=\ell_i,\\
\nu_i>0, && i\notin I(z^*)\mbox{ and } x_i^*(z_k)=u_i.
\end{eqnarray*}
This means that $I(z_k)=I(z^*)$. Similarly, we can consider the KKT conditions for \eqref{K} and can show that for large enough $k$ there holds $x(y_k,z_k)\to x(y^*,z^*)=x^*(z^*)$, and $I(y_k,z_k)=I(y^*,z^*)=I(z^*)$. This implies $I(y_k,z_k)=I(z^*)=I(z_k)$ for large $k$, a contradiction.
So next we assume that $I(y, z)=I(z)$. Also define $A_1=\{i\mid x_i^*(z)=x_i(y, z)=\ell_i\}$ and $A_2=\{i\mid x_i^*(z)=x_i(y, z)=u_i\}$.
Recall the optimization problem \eqref{K} for function $K$:
\begin{equation}\label{optK}
\begin{array}{ll}
\mbox{minimize }& \displaystyle f(x)+\frac{p}{2}\|x-z\|^2+y^T(Ax-b)+\frac{\Gamma}{2}\|Ax-b\|^2\\[5pt]
\mbox{subject to }&\ell_i\le x_i\le u_i, \ i=1, \cdots, n,
\end{array}
\end{equation}
and its KKT conditions:
\begin{eqnarray}
\nabla f(x(y, z))+p(x(y, z)-z)+A^Ty+\Gamma A^T(Ax(y,z)-b)-\mu(y, z)+&&\nu(y,z) = 0, \nonumber\\
\mu_i(y, z)\ge 0,\;\ell_i-x_i(y,z)=0, \quad&& \hfill i\in A_1,\nonumber\\
\nu_i(y, z)\ge 0,\; x_i(y, z)-u_i=0,\quad&& \hfill i\in A_2,\label{kktforK}\\
\mu_i(y, z)=0,\; \nu_i(y, z)=0,\quad && \hfill i\not\in A_1\cup A_2,\nonumber\\
x_i(y,z)\in [\ell_i, u_i],\quad && \hfill i\in I(z).\nonumber
\end{eqnarray}

Notice that the optimality conditions \eqref{kktforproximal} for the convex proximal problem \eqref{proximal} can be rewritten (using the information about the active and inactive sets) as:
\begin{eqnarray}
\nabla f(x^*(z))+p(x^*(z)-z)+A^Ty^*(z)-\mu(z)+\nu(z) = 0,\quad&& \hfill \nonumber\\
Ax^*(z)-b=0,\quad&& \hfill \nonumber\\
\mu_i(z)\ge 0,\;\ell_i-x_i^*(z)=0,\quad&& \hfill i\in A_1,\label{kktforproximallinear0}\\
\nu_i(z)\ge 0,\; (x_i^*(z)-u_i)=0,\quad&& \hfill i\in A_2,\nonumber\\
\mu_i(z),\; \nu_i(z)\ge 0,\; x_i^*(z)\in[\ell_1, u_1],\quad&& \hfill i\in I(z).\nonumber
\end{eqnarray}
Let $g(x; z)=f(x)+\frac{p}{2}\|x-z\|^2+\frac{\Gamma}{2}\|Ax-b\|^2$. Consider the following linear system.
\begin{eqnarray}\label{linear system}
A^Ty^*(z)-\mu(z)+\nu(z) = -\nabla g(x^*(z); z),\quad&& \hfill \nonumber\\
\mu_i(z)\ge 0,\;\ell_i-x_i^*(z)=0,\quad&& \hfill i\in A_1,\label{kktforproximallinear}\\
\nu_i(z)\ge 0,\; (x_i^*(z)-u_i)=0,\quad&& \hfill i\in A_2,\nonumber\\
\mu_i(z),\; \nu_i(z)\ge 0,\; x_i^*(z)\in[\ell_1, u_1],\quad&& \hfill i\in I(z).\nonumber
\end{eqnarray}
Notice that in the above system, the left-hand-side is linear of $y(z), \mu(z), \nu(z)$.
By \eqref{kktforK}, the vector $(x(y,z),y,\mu(y,z),\nu(y,z))$ satisfies \eqref{kktforproximal} approximately.
Using Hoffman bound (Proposition~\ref{hoffman}) to \eqref{kktforproximallinear} at the point $(y,\mu(y,z),\nu(y,z))$, we obtain
\begin{eqnarray*}
\dist(y, Y^*(z))^2&\le &\displaystyle \min_{(x^*(z), y(z),\mu(z),\nu(z))\atop \mbox{\scriptsize{ satisfying \eqref{kktforproximallinear}}}} 
\|y-y(z)\|^2+\|(\mu(y,z),\nu(y,z))-(\mu(z),\nu(z))\|^2\\
&\le & 2\theta^2\|\nabla g(x(y, z); z)-\nabla g(x^*(z); z)\|^2.\\ [3pt]
&\le& 2\theta^2\left((L+p+\Gamma\sigma^2)^2\|x(y, z)-x^*(z)\|^2\right), 
\end{eqnarray*}
where $\sigma$ is the largest singular value of $A$. Using Assumption~\ref{feasible-int}(d) and the strong convexity of $g(\cdot; z)$, we obtain
\begin{eqnarray}
\dist(y, Y\!&^*(&\!z))^2\le 2\theta^2\frac{(L+p+\Gamma\sigma^2)^2}{p+\gamma}\langle \nabla f(x(y,z))-\nabla f(x^*(z))+p(x(y,z)-x^*(z)), \nonumber\\
&&x(y, z)-x^*(z)\rangle.\nonumber\\
&=&-\frac{2\theta^2 (L+p+\Gamma\sigma^2)^2}{p+\gamma}\langle y-y(z),A(x(y, z)-x^*(z))\rangle\nonumber\\
&&+\frac{2\theta^2 (L+p+\Gamma\sigma^2)^2}{p+\gamma}\langle (\mu(y, z)-\mu(z))-(\nu(y, z)-\nu(z)), x(y, z)-x^*(z)\rangle \label{zero}\\
&\le& 2\theta^2\left(\frac{(L+p+\Gamma\sigma^2)^2}{p+\gamma}\|Ax(y, z)-b\|
\|y-y(z)\|\right),\label{key}
\end{eqnarray}
where the equality follows from \eqref{kktforK}-\eqref{kktforproximallinear} and the cross term \eqref{zero} vanishes because we have
\[
\begin{array}{rrr}
\mu_i(y, z)(x_i(y, z)-\ell_i)=0,&\quad
\mu_i(z)(x_i^*(z)-\ell_i)=0,& \quad i=1,2,...,n,\nonumber \\[5pt]
\mu_i(y, z)(x_i^*(z)-\ell_i)\ge0,&\quad
\mu_i(z)(x_i(y, z)-\ell_i)\ge 0,&\quad i=1,2,...,n, \nonumber \\[5pt]
\nu_i(y, z)(u_i-x_i(y, z))=0,&\quad
\nu_i(z)(u_i-x_i^*(z))=0,&\quad i=1,2,...,n,\nonumber\\[5pt]
\nu_i(y, z)(u_i-x_i^*(z))\ge0,&\quad
\nu_i(z)(u_i-x_i(y, z))\ge0,&\quad i=1,2,...,n.
\end{array}
\]
Hence we finish the proof of the dual error bound \eqref{eb5}.
\end{proof}
\noindent{\bf Remark}. Notice that the error bounds \eqref{eb1}, \eqref{eb2} hold for all strongly convex problems over a convex set when a linear term in the objective function is perturbed.
We now give the proof of Lemma \ref{proximal-descent}.

\begin{proof}[Proof of Lemma \ref{proximal-descent}]
First, using Danskin's theorem in convex analysis \cite{[25]}, we have:
$$\nabla M(z^t)=p(z^t-x^*(z^t)),$$
where $M(z)$ is defined by \eqref{P}.
So it suffices to prove that
$$\|\nabla M(z^t)-\nabla M(z^{t+1})\|\le ({\sigma}^{-1}_4+1)\|z^t-z^{t+1}\|.$$
But this is a direct corollary of the error bound \eqref{eb6} in Lemma \ref{error bound}. The proof is complete.
\end{proof}
\subsection{Proof of Theorem~\ref{th:main}}
\begin{proof}
Using the three descent lemmas in Subsection~\ref{sub:3descent}, we get
\begin{eqnarray}
&&\phi^t-\phi^{t+1}\nonumber\\
&\ge & (\frac{1}{2c}\|x^{t+1}-x^t\|^2-\alpha\|Ax^t-b\|^2+\frac{p}{2\beta}\|z^t-z^{t+1}\|^2)\nonumber\\
&&+2 \left(\alpha(Ax^t-b)^T(Ax(y^{t+1}, z^t)-b)+p(z^{t+1}-z^t)^T(z^{t+1}+z^t-2x(y^{t+1}, z^{t+1})) \right)\nonumber\\
&&+2 (p(z^{t+1}-z^t)^T(x^*(z^t)-z^t)-\frac{p\tilde{L}}{2}\|z^{t+1}-z^t\|^2 )\nonumber\\
&=& (\frac{1}{2c}\|x^{t+1}-x^t\|^2-\alpha\|Ax^t-b\|^2+\frac{p}{2\beta}\|z^t-z^{t+1}\|^2)
+2\alpha(Ax^t-b)^T(Ax(y^{t+1}, z^t)-b)\nonumber\\
&&+p(z^{t+1}-z^t)^T\left((z^{t+1}-z^t)-2(x(y^{t+1}, z^{t+1})-x^*(z^t))\right)
-\frac{p\tilde{L}}{2}\|z^t-z^{t+1}\|^2\nonumber\\
&=& {\color{black}(\frac{1}{2c}\|x^{t+1}-x^t\|^2-\alpha\|Ax^t-b\|^2+\frac{p}{2\beta}\|z^t-z^{t+1}\|^2)
+2\alpha(Ax^t-b)^T(Ax(y^{t+1}, z^t)-b)}\nonumber\\
&&+p(z^{t+1}-z^t)^T\left((z^{t+1}-z^t)-2(x(y^{t+1}, z^{t+1})-x(y^{t+1}, z^t))-2(x(y^{t+1}, z^t)-x^*(z^t))\right)\nonumber\\
&&-\frac{p\tilde{L}}{2}\|z^t-z^{t+1}\|^2.\nonumber\\
\label{first descent}
\end{eqnarray}
Let $\lambda$ be an arbitrary positive scalar, and by the fact that
$$\|(z^{t+1}-z^t)/\lambda-\lambda(x(y^{t+1}, z^t)-x^*(z^t))\|^2\ge 0,$$
we have
$$-2(z^{t+1}-z^t)^T(x(y^{t+1}, z^t)-x^*(z^t))\ge -\|z^t-z^{t+1}\|^2/\lambda-\lambda\|x(y^{t+1}, z^{t})-x^*(z^t)\|^2.$$
Using Cauchy-Schwarz inequality and the error bound \eqref{eb3} in Lemma~\ref{error bound}, we have
\begin{eqnarray*}
-2(z^{t+1}-z^t)^T(x(y^{t+1}, z^{t+1})-x(y^{t+1}, z^t))&\ge& -\|z^t-z^{t+1}\|\|x(y^{t+1}, z^{t+1})-x(y^{t+1}, z^t)\|\\
&\ge& -\frac{1}{\sigma_4}\|z^t-z^{t+1}\|^2.
\end{eqnarray*}
Substituting these two inequalities into \eqref{first descent}, we have
\begin{eqnarray*}
&&\phi^t-\phi^{t+1}\nonumber\\
&\ge & \frac{1}{2c}\|x^{t+1}-x^t\|^2-\left(\alpha\|Ax^t-b\|^2-2\alpha(Ax^t-b)^T(Ax(y^{t+1}, z^t)-b)+\alpha\|Ax(y^{t+1}, z^t)-b\|^2\right)\nonumber\\
&&+\alpha\|Ax(y^{t+1}, z^t)-b\|^2
+(\frac{p}{2\beta}+p-\frac{p}{\sigma_4}-\frac{p}{\lambda}-\frac{p\tilde{L}}{2})\|z^t-z^{t+1}\|^2\nonumber\\
&&-p\lambda\|x(y^{t+1}, z^t)-x^*(z^t)\|^2.
\end{eqnarray*}
By completing the square, we further obtain
\begin{eqnarray}
\phi^t-\phi^{t+1}
& \ge& \frac{1}{2c}\|x^{t+1}-x^t\|^2-\alpha\|A(x(y^{t+1}, z^t)-x^t)\|^2
+\alpha\|Ax(y^{t+1}, z^t)-b\|^2\nonumber\\
&&+(\frac{p}{2\beta}+p-\frac{p}{\sigma_4}-\frac{p}{\lambda}-\frac{p\tilde{L}}{2})\|z^t-z^{t+1}\|^2
-p\lambda\|x(y^{t+1}, z^t)-x^*(z^t)\|^2\nonumber\\[3pt]
&\ge& (\frac{1}{2c}-\frac{\alpha\sigma^2}{ \sigma_1^2})\|x^t-x^{t+1}\|^2+\alpha\|Ax(y^{t+1}, z^t)-b\|^2
\nonumber\\
&&
+(\frac{p}{2\beta}+p-\frac{p}{\sigma_4}-\frac{p}{\lambda}-\frac{p\tilde{L}}{2})\|z^t-z^{t+1}\|^2
-p\lambda\|x(y^{t+1}, z^t)-x^*(z^t)\|^2, \label{one}
\end{eqnarray}
where $\sigma$ is the spectral norm of $A$, $\lambda$ is any positive scalar, and the last step is due to the error bound \eqref{eb1} in Lemma~\ref{error bound}.
Let $\lambda=D\beta$ for some sufficiently large $D$ (for example, $D>6$), and set $\beta$ sufficiently small ($\beta\le \beta'$) for some constant $\beta'>0$, such that
$$\frac{p}{2\beta}+p-\frac{p}{\sigma_4}-\frac{p}{\lambda}-\frac{p\tilde{L}}{2}\ge \frac{p}{3\beta}.$$
Therefore, if we choose $\alpha< \frac{\sigma^{2}_1}{4c\sigma^2 }$, 
then it follows from \eqref{one} that
\begin{eqnarray}
\phi^t-\phi^{t+1}&\ge& \frac{1}{4c}\|x^t-x^{t+1}\|^2+\alpha\|Ax(y^{t+1}, z^t)-b\|^2\nonumber\\
&&+\frac{p}{3\beta}\|z^t-z^{t+1}\|^2-pD\beta\|x(y^{t+1}, z^t)-x^*(z^t)\|^2. \label{key1}
\end{eqnarray}
It remains to bound the term $pD\beta\|x(y^{t+1}, z^t)-x^*(z^t)\|^2$ in the above expression. The main proof idea is as follows.
In view of the dual error bound \eqref{eb5}, when the residuals are sufficiently small, we can use the dual residual $\|Ax(y^{t+1}, z^t)-b\|$ to bound $\dist(y^{t+1}, Y^*(z^t))$ and then further use the error bound \eqref{eb6} to bound $\|x(y^{t+1}, z^t)-x^*(z^t)\|$. When some residual is not too small, we will make use of the compactness of the feasible set to bound the term $\|x(y^{t+1}, z^t)-x^*(z^t)\|^2$.
Define
\begin{equation}\label{def1}
M=\max_{x_1, x_2\in P}\|x_1-x_2\|,\quad \zeta=\min\{\Delta, \delta(\Delta/\sqrt{6D})\}
\end{equation}
and set
\begin{equation}\label{def2}
\beta<\min\left\{\beta', \frac{\zeta^2}{8cpDM^2}, \frac{\zeta^2\alpha}{2pDM^2}, \frac{\alpha\sigma_3^2}{2pD\sigma_5^{2}}\right\},
\end{equation}
where $\Delta$ is defined in Lemma~\ref{error bound} and $\delta(\cdot)$ is defined in Corollary \ref{coro1} and Corollary~\ref{coro2}.
We also define the following three conditions:
\begin{eqnarray}\label{cond1}
\|x^t-x^{t+1}\|^2\le 8cpDM^2\beta, \\
\|Ax(y^{t+1}, z^t)-b\|^2\le \frac{2pDM^2}{\alpha} \beta,\label{cond2}\\
\|x^{t+1}-z^t\|^2=\|(z^t-z^{t+1})/\beta\|^2\le 6D\|x(y^{t+1}, z^t)-x^*(z^t)\|^2.\label{cond3}
\end{eqnarray}
We now consider two cases.\\
\noindent {\bf Case 1.} Conditions \eqref{cond1}-\eqref{cond3} hold. In this case, it follows from \eqref{def1}-\eqref{def2} that
\begin{eqnarray*}
\|x^t-x^{t+1}\|\le \zeta = \min\{\Delta, \delta(\Delta/\sqrt{6D})\}\le \Delta,\\
\|Ax(y^{t+1}, z^t)-b\|\le \zeta = \min\{\Delta, \delta(\Delta/\sqrt{6D})\}\le \Delta,
\end{eqnarray*}
which further implies
\[
\|x^{t+1}-z^t\|=\|(z^t-z^{t+1})/\beta\|\le \sqrt{6D}\|x(y^{t+1}, z^t)-x^*(z^t)\|\le \sqrt{6D}\frac{\Delta}{\sqrt{6D}}=\Delta,
\]
where the last inequality follows from Corollary~\ref{coro2}.
Therefore, the error bounds \eqref{eb6} and \eqref{eb5} in Lemma~\ref{error bound} hold and we have
\begin{eqnarray}\label{eq:1}
pD\beta\|x(y^{t+1}, z^t)-x^*(z^t)\|^2
&\le & pD\beta\sigma_3^{-2}\cdot\dist(y^{t+1}, Y^*(z^t))^2 \nonumber\\
&\le & pD\beta\sigma_3^{-2}\sigma_5^2\|Ax(y^{t+1}, z^t)-b\|^2 \nonumber\\
&\le & \frac{\alpha}{2}\|Ax(y^{t+1}, z^t)-b\|^2,
\end{eqnarray}
where the last step follows from \eqref{def2}.
It then follows from \eqref{key1} that
\begin{eqnarray}\label{key3}
\phi^t-\phi^{t+1}
&\ge& \frac{1}{4c}\|x^t-x^{t+1}\|^2+\frac{\alpha}{2}\|Ax(y^{t+1}, z^t)-b\|^2+\frac{p}{3\beta}\|z^t-z^{t+1}\|^2.
\end{eqnarray}
{\bf Case 2.} One of the conditions \eqref{cond1}-\eqref{cond3} is violated. Consider three subcases:
\begin{enumerate}
\item [] \underline{Case 2.1.} $\|x^t-x^{t+1}\|^2\ge 8cpD\beta M^2$. In this case, we have
\begin{eqnarray}\label{eq:2.1}
\frac{1}{4c}\|x^t-x^{t+1}\!\!&\|^2&\!\!-pD\beta \|x(y^{t+1}, z^t)-x^*(z^t)\|^2 \nonumber \\
&\ge & \frac{1}{8c}\|x^t-x^{t+1}\|^2+\frac{1}{8c}{8cpD\beta M^2}-pD\beta \|x(y^{t+1}, z^t)-x^*(z^t)\|^2 \nonumber\\
&\ge& \frac{1}{8c}\|x^t-x^{t+1}\|^2+{pD\beta M^2}-pD\beta M^2 \nonumber\\
&= & \frac{1}{8c}\|x^t-x^{t+1}\|^2,
\end{eqnarray}
where the second step is due to \eqref{def1}.
\item [] \underline{Case 2.2.} $\|Ax(y^{t+1}, z^t)-b\|^2\ge \frac{2pD \beta}{\alpha}M^2$. In this case, we use this condition and \eqref{def1} to obtain
\begin{eqnarray}\label{eq:2.2}
\alpha\|Ax(y^{t+1}, z^t)-b\|^2\!\!&-&\!\! pD\beta\|x(y^{t+1}, z^t)-x^*(z^t)\|^2 \nonumber\\
&\ge &\frac{\alpha}{2}\|Ax(y^{t+1}, z^t)-b\|^2+\frac{\alpha}{2}\cdot\frac{2pD \beta}{\alpha}M^2-pD\beta M^2 \nonumber\\
&=& \frac{\alpha}{2}\|Ax(y^{t+1}, z^t)-b\|^2.
\end{eqnarray}
\item []\underline{Case 2.3.} $\|(z^t-z^{t+1})/\beta\|^2\ge 6D\|x(y^{t+1}, z^t)-x^*(z^t)\|^2$. In this case, we have
\begin{eqnarray}\label{eq:2.3}
\frac{p}{3\beta}\|z^t-z^{t+1}\|^2\!\!&-&\!\! pD\beta\|x(y^{t+1}, z^t)-x^*(z^t)\|^2 \nonumber\\
&\ge&\frac{p}{3\beta}\|z^t-z^{t+1}\|^2-pD\beta \|x(y^{t+1}, z^t)-x^*(z^t)\|^2 \nonumber\\
&\ge & \frac{p}{6\beta}\|z^t-z^{t+1}\|^2+\frac{p}{6\beta}6\beta^2D\|x(y^{t+1}, z^t)-x^*(z^t)\|^2 \nonumber\\
&&-pD\beta \|x(y^{t+1}, z^t)-x^*(z^t)\|^2 \nonumber\\
&= & \frac{p}{6\beta}\|z^t-z^{t+1}\|^2.
\end{eqnarray}
\end{enumerate}
Considering \eqref{key1}, we have in all three subcases:
\begin{equation*}
\phi^t-\phi^{t+1}
\ge \frac{1}{8c}\|x^t-x^{t+1}\|^2+\frac{\alpha}{2}\|Ax(y^{t+1}, z^t)-b\|^2+\frac{p}{6\beta}\|z^t-z^{t+1}\|^2.
\end{equation*}
Combining \eqref{eq:1}, \eqref{eq:2.1}, \eqref{eq:2.2} and \eqref{eq:2.3} yields
\begin{equation}\label{dual bound for x}
pD\beta\|x(y^{t+1}, z^t)-x^*(z^t)\|^2\le \frac{1}{8c}\|x^t-x^{t+1}\|^2+\frac{\alpha}{2}\|Ax(y^{t+1}, z^t)-b\|^2+\frac{p}{6\beta}\|z^t-z^{t+1}\|^2.
\end{equation}
Then according to \eqref{key3}, we have
\begin{equation}\label{potential-descent}
\phi^t-\phi^{t+1}\ge\frac{1}{8c}\|x^t-x^{t+1}\|^2+\frac{\alpha}{2}\|Ax(y^{t+1}, z^t)-b\|^2+\frac{p}{6\beta}\|z^t-z^{t+1}\|^2,\quad \forall\; t\ge0.
\end{equation}
Since $\phi^t$ is bounded below, we must have
$$\max\{\|x^{t+1}-x^t\|, \|Ax(y^{t+1}, z^t)-b\|, \|z^t-x(y^{t+1}, z^t)\|\}\to 0.$$
This together with Corollary \ref{coro1} shows that the KKT condition for \eqref{P} is satisfied in the limit. This completes the proof.
\end{proof}

Notice that we can take $D=7$, then the threshold for the stepsize $\beta$ is given in \eqref{def2}, where $\zeta$ and $M$ are defined in \eqref{def1}. Since $\zeta$ is defined by constants $\Delta,\ \delta(\Delta)$ from the  ``local dual error bound'', it is difficult to derive an explicit bound for the stepsize $\beta$. In practice, we will need to tune the stepsize $\beta$ to ensure fast convergence. To give an explicit upper-bound for $\beta$, a ``global'' error bound is needed. This will be the focus of a new paper under planning.

Theorem~\ref{th:main} establishes the global convergence of Algorithm~\ref{Alg2} to a stationary solution. However, it does not
address the question of convergence rate. The latter is considered in the next section. We will show that the iteration complexity is $O(1/\epsilon^2)$ for a fixed problem and
for the special case when the objective function is (nonconvex) quadratic the convergence rate is linear. 

\section{Iteration Complexity and Linear Convergence}
\subsection{$O(1/\epsilon^2)$ iteration complexity}
In this subsection, we will see that the iteration complexity of Algorithm~\ref{Alg2} is $O(1/\epsilon^2)$ to attain an $\epsilon$-stationary solution.
We first define the $\epsilon$-stationary solution as in \cite{Monteiro}. Let $\iota(x)$ be the indicate function of the set $P$, i.e. , $\iota(x)=0$ if $x\in P$ and $\iota(x)=\infty$ otherwise.
\begin{definition}
We say that $(x, y)$ is an $\epsilon$-solution of \eqref{P} if $\|Ax-b\|\le \epsilon$ and  there exists a vector  $v\in \nabla f(x)+A^Ty+\partial{\iota(x)}$ with $\|v\|\le \epsilon$.
\end{definition}
\begin{theorem}
There exists a constant $B>0$ such that  for any $t>0$, we can find an $s\in \{0, 1, \cdots, t-1\}$ such that $(x^{s+1}, y^{s+1})$ is a $B/\sqrt{t}$-solution. In other word, we can find an $\epsilon$- solution within $B^2/\epsilon^2$ iterations.
\end{theorem}
\noindent{\bf Remark.}Note that the constant $B$ depends on some ``local'' constants and error bounds related to the local geometry of the problem.  Global error bounds are needed for a global complexity analysis, which depends on global constants. This interesting topic will be left as future work in our follow-up paper.
\begin{proof}
According to \eqref{phi} , for any $t>0$, there exists an $s\in\{0, 1, \cdots, t-1\}$ such that$\phi^s-\phi^{s+1}\le \frac{\phi^0-\underline{f}}{t}$.  Then by inequality \eqref{potential-descent}, letting
$$C=(\phi^0-\underline{f})\cdot \max\{8c, 2/\alpha, 6\beta/p\},$$
 we have
\begin{eqnarray}\label{eq:4.3}
\|x^s-x^{s+1}\|^2&<&C/t;\\
\|Ax(y^{s+1}, z^s)-b\|^2&<&C/t;\\
\|x^{s+1}-z^s\|^2&<&C/t,
\end{eqnarray}
According to Algorithm \ref{Alg2}, we have
$$x^{s+1}= \arg\min_{x}\{\langle \nabla_xK(x^s, z^s; y^{s+1}), x-x^s \rangle+\frac{1}{c}\|x-x^s\|^2+\iota(x)\}.$$
Hence, by the optimality condition,
$$0\in \nabla_xK(x^s, z^s;  y^{s+1})+\frac{2}{c}(x^{s+1}-x^s)+\partial{\iota(x^{s+1})}.$$
Letting
$$v=\nabla_xK(x^{s+1}, z^s; y^{s+1})-\nabla_xK(x^s, z^s; y^{s+1})-\frac{2}{c}(x^{s+1}-x^s)-\Gamma(Ax^{s+1}-b)-p(x^{s+1}-z^s),$$
we have
\begin{eqnarray}
v\in \nabla_x(f(x^{s+1}+A^Ty^{s+1}))+\partial{(\iota(x^{s+1}))}.
\end{eqnarray}
By inequalities \eqref{eq:4.3} and \eqref{eb2}, we have
\begin{eqnarray*}\label{eq:4.4}
\|Ax^{s+1}-b\|&\le&\|Ax(y^{s+1}, z^s)-b\|+\|A(x^{s+1}-x(y^{s+1}, z^s))\|\\
&\le&\sqrt{C}/\sqrt{t}+\sigma\sigma_2^{-1} \sqrt{C}/\sqrt{t}\\
&\le&\sqrt{C}(1+\sigma\sigma_2^{-1})/\sqrt{t},
\end{eqnarray*}
where the first inequality is because of triangle inequality. 
Then we have
\begin{eqnarray*}
\|v\|&\le& (L+p+\Gamma\sigma^2+2/c)\|x^s-x^{s+1}\|+\Gamma\|Ax^{s+1}-b\|+p\|x^{s+1}-z^s\|\\
&\le& (L+p+\Gamma\sigma^2+2/c)\sqrt{C}/\sqrt{t}+\Gamma(\sqrt{C}/\sqrt{t}+\sqrt{C}\sigma\sigma_2^{-1}/\sqrt{t})+p\sqrt{C}/\sqrt{t}\\
&\le&\sqrt{B}/\sqrt{t},
\end{eqnarray*}
where the first inequality is due to the triangle inequality and the Lipschitz continuity of $\nabla_xK(\cdot, z; y)$, the second inequality is because of inequalities \eqref{eq:4.3}, \eqref{eq:4.4} and
$$B=((L+p+\Gamma\sigma^2+2/c)+\Gamma(1+\sigma\sigma_2^{-1})+p)^2C.$$
Notice that $B>0$. Then the result holds for $v$ and $B$ and $(x^{s+1}, y^{s+1})$ is a $B/\sqrt{t}$-solution.

\end{proof}
\subsection{Linear convergence for quadratic programming}
In this subsection, we consider a nonconvex quadratic program (QP), which is a special case of \eqref{P} with $f(x)$ being a quadratic function
\begin{equation}\label{quadratic}
f(x)=\frac{1}{2}x^TQx+r^Tx.
\end{equation}
We will strengthen Theorem~\ref{th:main} in this case by showing that Algorithm~\ref{Alg2} converges linearly to a stationary point of the nonconvex QP problem.
{\textcolor{black}By Theorem~\ref{th:main}, we have $\dist(x^t,X^*)\to 0$, $\dist(z^t, X^*)\to 0$ and $\|x^{t+1}-x^t\|\to 0$ as $t\to\infty$. Since $X^*$ is the union a finite number of polyhedral sets, it follows that the connected components of $X^*$ are properly separated in the sense that there is a positive distance between each pair of distinct connected components of $X^*$. As a result, the sequences $\{x^t\}$ and $\{z^t\}$ will both converge to one unique connected component of $X^*$. Moreover, it is known \cite{eb-survey} that for a quadratic programming problem, the objective function value $f(x)$ is constant on each of the connected component of $X^*$. Let $f^*$ denote the value of $f(x)$ over the connected component of $X^*$ to which $x^t$ (and $z^t$) converges. Then $f(x^t)\to f^*$ as $t\to \infty$.}
Note that by the analysis above, when $t$ is sufficiently large, $f(\bar{z}^t)=f^*$ and $\bar{z}^t$ belongs to the connected component that $x^t$ and $z^t$ converge to.
We summarize the above analysis as follows.
\begin{claim}\label{claim}
{\color{black}Assume the parameters of Algorithm \ref{Alg2} are chosen to guarantee its convergence $($cf.\ Theorem~\ref{th:main}$)$.}
We have the following.
\begin{enumerate}
\item The quadratic cost function $f$ is constant on each of the connected components of $X^*$.
\item The two sequences $\{x^t\}$, $\{z^t\}$ converge to a same connected component of $X^*$.
\item Let $\displaystyle\bar{z}^t=\argmin_{\bar{z}\in X^*}=\dist({z}^t,X^*)$ and $x^*$ is any limit point of $\{x^t\}$, then $f(x^*)=f(\bar{z}^t)=f^*$, where $f^*$ is a constant, for all sufficiently large $t$.
\item There exists a constant $\epsilon(x^0, y^0, z^0)$, such that for any $z^t$, if $\dist(z^t, X^*)< \epsilon(x^0, y^0, z^0)$, we always have
$$f(\bar{z}^t)=f^*.$$
\end{enumerate}
\end{claim}
Denote
\begin{eqnarray}
\Phi_p^t(z)&=&K(x^t, z; y^t)-d(y^t, z)\nonumber \\
\Phi_d^t(z)&=&M(z)-d(y^t, z) \nonumber\\
\Phi_{prx}^t(z)&=&M(z)-f^*\label{delta-prx}\\
\bar{z}^t&=&\argmin_{x\in X^*}\|x-z^t\|.\nonumber\\
\Phi^t(z)&=&\Phi^t_{p}(z)+\Phi^t_d(z)+\Phi^t_{prx}(z).\nonumber
\end{eqnarray}
\noindent{\bf Remark.} Notice that $\Phi^t(z^t)$ decreases monotonically due to \eqref{potential-descent} and converges to $0$ by Claim \ref{claim}, we have $\Phi^t(z^t)=\phi^t-{f^*}$ and $\Phi^t(z^t)\ge 0$ for any $t\ge 0$.
Also notice that
\begin{equation}\label{equivalent}
\phi^t-\phi^{t+1}=\Phi^t(z^t)-\Phi^{t+1}(z^{t+1}).
\end{equation}
To prove linear convergence, we make use of some ``cost-to-go'' estimates from \cite{Luo-Hong12}.
\begin{lemma}\label{cost-to-go}
There exist constants $\tau_1, \tau_2, \tau_3>0$ such that
\begin{eqnarray}\label{primal-cost-to-go}
K(x^{t+1}, z^t; y^{t+1})-d(y^{t+1}, z^t)&\le& \tau_1\|x^t-x^{t+1}\|^2,\\
\label{dual-cost-to-go}
M(z^t)-d(y^{t+1}, z^t)&\le& \tau_2\|Ax(y^{t+1}, z^t)-b\|^2,\\
\label{proximal-cost-to-go}
M(z^t)-f^*&\le& \tau_3\|z^t-x^*(z^t)\|^2.
\end{eqnarray}
\end{lemma}
\begin{proof}
The proof of {\color{black}\eqref{primal-cost-to-go} and \eqref{dual-cost-to-go} is simply to combine Lemma 3.1 of \cite{Luo-Hong12} with \eqref{eb1}, \eqref{eb5} in Lemma~\ref{error bound}. Specifically, we only need to replace $L(x; y)$, $d(y)$, $p^*$ of \cite[Lemma 3.1]{Luo-Hong12} by $K(x, z;y)$, $d(y, z)$, $M(z)$ respectively with $z$ fixed, since the error estimates of \cite{Luo-Hong12} are independent of the linear term in $K(x,z;y)$.}
To prove the estimate \eqref{proximal-cost-to-go}, we first notice that $f$ has a Lipschitz continuous gradient and that the classic proximal algorithm belongs to the class of approximate gradient projection algorithm (see \cite[Theorem 3.3]{eb-survey}), namely, we have
$$x^*(z^t)=[z^t-\nabla M(z^t)+\Theta(t)]_+,$$
where $\Theta(t)$ satisfies $\|\Theta(t)\|\le \eta\|x^*(z^t)-z^t\|$ for some $\eta>0$. Note that the above are similar to the inequalities (3.5) and (3.7) in the proof of the second part of Lemma 3.1 of \cite{Luo-Hong12}.
Therefore, similar to \cite[Lemma~ 3.1]{Luo-Hong12}, there exists a constant $\tau'$ such that
\begin{equation*}
f(x^*(z^t))-f(\bar{z}^t)
\le \tau'(\|z^t-x^*(z^{t})\|^2+\|z^t-\bar{z}^t\|^2),
\end{equation*}
where $\bar{z}^t$ is the projection of $z^t$ to stationary solution set of problem \eqref{P}. According to Claim \ref{claim}, $\bar{z}^t$ is in the connected component that $\{x^t\}$, $\{z^t\}$ converge to and $f(\bar{z}^t)=f^*$. 
From Theorem 2.1 in \cite{eb-survey}, there exists a constant $\bar{\tau}>0$, such that
$$\|z^t-\bar{z}^t\|\le \bar{\tau}\|z^t-x^*(z^t)\|.$$
Combining the above two inequalities and using the definition of $M(z)$, we have
\begin{equation*}
M(z^t)-f^*=f(x^*(z^t))+\frac{p}{2}\|x^*(z^t)-z^t\|^2-f(\bar{z}^t)
\le \left((\bar{\tau}^2+1)\tau'+\frac{p}{2}\right)\|z^t-x^*(z^t)\|^2. 
\end{equation*}
This completes the proof of \eqref{proximal-cost-to-go} by setting $\tau_3=\left((\bar{\tau}^2+1)\tau'+\frac{p}{2}\right)$.
\end{proof}
\bigskip
The following is the error bound for proximal algorithm, which can be seen in \cite{eb-survey}.
\begin{lemma}\label{proximal error bound}
There exist a constant $\bar{\tau}$ such that
$$\dist(z, x^*)\leq \bar{\tau}\|z-x^*(z)\|.$$
\end{lemma}
\begin{lemma}\label{proximal residual}
For any $t\ge 0$, we have
$$\|z^t-x^*(z^t)\|^2\le B_1(\beta)\|z^t-z^{t+1}\|^2+B_2(\beta)\|x^t-x^{t+1}\|^2+B_3(\beta)\|Ax(y^{t+1}, z^t)-b\|^2,$$
where
$$B_1(\beta)=3(\frac{1}{\beta^2}+\frac{1}{6D\beta^2}),\quad B_2(\beta)=3(\frac{1}{\sigma_2^2}+\frac{1}{8cpD\beta}) \ \mbox{ and }\ B_3(\beta)=\frac{3\alpha}{2pD\beta}.
$$
\end{lemma}
\begin{proof}
Notice that
\begin{eqnarray}\label{decomposition}
&&\|z^t-x^*(z^t)\|^2\nonumber\\
&=&\|(z^t-x^{t+1})+(x^{t+1}-x(y^{t+1}, z^t))+(x(y^{t+1}, z^t)-x^*(z^t))\|^2\\
&\stackrel{ \mbox{\scriptsize(i)}}=&\|(z^t-z^{t+1})/\beta+(x^{t+1}-x(y^{t+1}, z^t))+(x(y^{t+1}, z^t)-x^*(z^t))\|^2\\
&\stackrel{ \mbox{\scriptsize(ii)}}\le&3(\|z^t-z^{t+1}\|^2/\beta^2+\|x^{t+1}-x(y^{t+1}, z^t)\|^2+\|x(y^{t+1}, z^t)-x^*(z^t)\|^2)\\
&\stackrel{ \mbox{\scriptsize(iii)}}\le&3(\|z^t-z^{t+1}\|/\beta^2+\|x^t-x^{t+1}\|^2/\sigma_2^2+\|x(y^{t+1}, z^t)-x^*(z^t)\|^2),
\end{eqnarray}
where (i) is  because $z^{t+1}=z^t+\beta(x^{t+1}-z^t)$, (ii) is due to the convexity of the square norm and (iii) is due to \eqref{eb2}.
On the other hand, by \eqref{dual bound for x}, we have
\begin{equation}\label{dual bound for x 2}
\|x(y^{t+1}, z^t)-x^*(z^t)\|^2\le \frac{1}{pD\beta}(\frac{1}{8c}\|x^t-x^{t+1}\|^2+\frac{\alpha}{2}\|Ax(y^{t+1}, z^t)-b\|^2+\frac{p}{6\beta}\|z^t-z^{t+1}\|^2).
\end{equation}
Finally, we substitute \eqref{dual bound for x 2} to \eqref{decomposition} and get the desired result.
\end{proof}
\begin{lemma}\label{descent}
For $t>0$ and $\beta$ sufficiently small, we have
$$\phi^t-\phi^{t+1}\ge C_1(\beta)\|x^t-x^{t+1}\|^2+C_2(\beta)\|Ax(y^{t+1}, z^t)-b\|^2+C_3(\beta)\|z^t-x^*(z^t)\|^2+C_4(\beta)\|z^t-z^{t+1}\|^2,$$
where $C_i(\beta)>0, i=1, 2, 3, 4$ are constants depending on $\beta$.
\end{lemma}
\begin{proof}
By Lemma \ref{proximal residual}, we have
\begin{equation}
\|z^t-x^*(z^t)\|^2\le B_1(\beta)\|z^t-z^{t+1}\|^2+B_2(\beta)\|x^t-x^{t+1}\|^2+B_3(\beta)\|Ax(y^{t+1}, z^t)-b\|^2.
\end{equation}
Let $\bar{B}=\max\left\{\frac{B_1(\beta)}{1/8c}, \frac{B_2(\beta)}{\alpha/2}, \frac{B_3(\beta)}{p/6\beta}\right\}$. We then have
\begin{eqnarray}
\|z^t-x^*(z^t)\|^2/(2\bar{B})&\le& \frac{1}{16c}\|x^t-x^{t+1}\|^2+\frac{\alpha}{4}\|Ax(y^{t+1}, z^t)-b\|^2+\frac{p}{12\beta}\|z^t-z^{t+1}\|^2\nonumber\\
&\le&\frac{1}{2}(\phi^t-\phi^{t+1}).\label{barB2}
\end{eqnarray}
Hence, we have
\begin{eqnarray*}
\phi^t-\phi^{t+1}
&\ge&\frac{1}{8c}\|x^t-x^{t+1}\|^2+\frac{\alpha}{2}\|Ax(y^{t+1}, z^t)-b\|^2+\frac{p}{6\beta}\|z^t-z^{t+1}\|^2\\
&\ge&\frac{1}{16c}\|x^t-x^{t+1}\|^2+\frac{\alpha}{4}\|Ax(y^{t+1}, z^t)-b\|^2+\frac{p}{12\beta}\|z^t-z^{t+1}\|^2+\|z^t-x^*(z^t)\|^2/(2\bar{B}).
\end{eqnarray*}
This completes the proof.
\end{proof}
Now we can prove the linear convergence of Algorithm \ref{Alg2} for quadratic programming.
\begin{lemma}\label{one-step-linear-convergence}
If $f(\bar{z}^t)=f^*$ and $\max\{\dist(z^t, X^*), \|Ax(y^{t+1}, z^t)-b\|\}<\Delta$, we have
$$\Phi^t(z^t)-\Phi^{t+1}(z^{t+1})\ge \kappa\Phi^{t+1}(z^{t+1}),$$
where $\Delta$ is defined in Lemma \ref{error bound} and $\kappa>0$ is a constant. Hence, $\Phi^{t+1}(z^{t+1})<\frac{1}{1+\kappa}\Phi^t(z^t)$.
\end{lemma}

\begin{proof}
We need to relate the gap $\Phi^t(z^t)=\Phi_p^t(z^t)+\Phi_d^t(z^t)+\Phi_{prx}^t(z^t)$ to the estimates in Lemma~\ref{cost-to-go}.
Then using Lemmas~\ref{primal}-\ref{proximal-descent} and Lemma \ref{cost-to-go}, we have
\begin{eqnarray}
\Phi_p^{t+1}(z^t)&\le&\tau_1\|x^t-x^{t+1}\|^2,\label{eq1}\\
\Phi_d^{t+1}(z^t)&\le&\tau_2\|Ax(y^{t+1}, z^t)-b\|^2,\label{eq2}\\
\Phi_{prx}^{t+1}(z^t)&\le&\tau_3\|z^t-x^*(z^t)\|^2.\label{eq3}
\end{eqnarray}
Recall the inequality~\eqref{zprimal} in the proof of Lemma~\ref{primal}
$$K(x^{t+1}, z^{t+1};y^{t+1})-K(x^{t+1}, z^t; y^{t+1})\le -\frac{p}{2\beta}\|z^t-z^{t+1}\|^2,$$
and the inequality~\eqref{zdual}
$$d(y^{t+1}, z^t)-d(y^{t+1}, z^{t+1})\le -\frac{p}{2}(z^{t+1}-z^t)^T(z^{t+1}+z^t-2x(y^{t+1}, z^{t+1})).$$
Then for ${\Phi}_p^{t+1}(z^t)$, we have
\begin{eqnarray*}
\Phi_p^{t+1}(z^{t+1})-{\Phi}_p^{t+1}(z^t)&=&( K(x^{t+1}, z^{t+1};y^{t+1})-K(x^{t+1}, z^t; y^{t+1}))+( d(y^{t+1}, z^t)-d(y^{t+1}, z^{t+1}))\\
&\le& -\frac{p}{2\beta}\|z^t-z^{t+1}\|^2-\frac{p}{2}(z^{t+1}-z^t)^T(z^{t+1}+z^t-2x(y^{t+1}, z^{t+1})).
\end{eqnarray*}
For ${\Phi}_d^{t+1}(z^t)$, we have
\begin{eqnarray*}
\Phi_d^{t+1}(z^{t+1})-{\Phi}_d^{t+1}(z^t)&=&d(y^{t+1}, z^t)-d(y^{t+1}, z^{t+1})+M(z^{t+1})-M(z^t)\\
&\le& -\frac{p}{2}(z^{t+1}-z^t)^T(z^{t+1}+z^t-2x(y^{t+1}, z^{t+1}))+M(z^{t+1})-M(z^t).
\end{eqnarray*}
Also, it follows from the definition \eqref{delta-prx} that
\[
\Phi_{prx}^{t+1}(z^{t+1})=\Phi_{prx}^t(z^t)+M(z^{t+1})-M(z^t).
\]
Combining the above three inequalities, we have
\begin{eqnarray*}
\Phi_p^{t+1}(z^{t+1})\!\!&+&\!\!\Phi_d^{t+1}(z^{t+1})+\Phi_{prx}^{t+1}(z^{t+1})\\
&\le& \Phi_p^{t+1}(z^t)+{\Phi}_d^{t+1}(z^t)+{\Phi}_{prx}^{t+1}(z^t)\\
&&-\frac{p}{2\beta}\|z^t-z^{t+1}\|^2-\frac{p}{2}(z^{t+1}-z^t)^T(z^{t+1}+z^t-2x(y^{t+1}, z^{t+1}))\\
&&-\frac{p}{2}(z^{t+1}-z^t)^T(z^{t+1}+z^t-2x(y^{t+1}, z^{t+1}))+2\left(M(z^{t+1})-M(z^t)\right)\\
&=& {\Phi}_p^{t+1}(z^t)+{\Phi}_d^{t+1}(z^t)+{\Phi}_{prx}^{t+1}(z^t)\\
&&-\frac{p}{2\beta}\|z^t-z^{t+1}\|^2-p(z^{t+1}-z^t)^T(z^{t+1}+z^t-2x(y^{t+1}, z^{t+1}))\\
&&+2p(z^{t+1}-z^t)^T(z^t-x^*(z^t))+\frac{p}{{\sigma}_4}\|z^t-z^{t+1}\|^2\\
&=& \Phi_p^{t+1}(z^t)+{\Phi}_d^{t+1}(z^t)+{\Phi}_{prx}^{t+1}(z^t)\\
&&+(-\frac{p}{2\beta}+\frac{p}{{\sigma}_4}-p)\|z^t-z^{t+1}\|^2+2p(z^{t+1}-z^t)^T\left(x(y^{t+1}, z^{t+1})-x^*(z^t)\right),
\end{eqnarray*}
where the second step follows from \eqref{eq:prox-descent}.
Using the inequality $2a^Tb\le \|a\|^2+\|b\|^2$ for any $a,b\in\mathbb{R}^n$, we have
\begin{eqnarray*}
\Phi_p^{t+1}(z^{t+1})\!\!&+&\!\!\Phi_d^{t+1}(z^{t+1})+\Phi_{prx}^{t+1}(z^{t+1})\\
&\le& \Phi_p^{t+1}(z^t)+{\Phi}_d^{t+1}(z^t)+{\Phi}_{prx}^{t+1}(z^t)\\
&&+(-\frac{p}{2\beta}+\frac{p}{{\sigma}_4}-p)\|z^t-z^{t+1}\|^2+p\|z^t-z^{t+1}\|^2+p
\|x(y^{t+1}, z^{t+1})-x^*(z^t)\|^2\nonumber\\
&=& \Phi_p^{t+1}(z^t)+{\Phi}_d^{t+1}(z^t)+{\Phi}_{prx}^{t+1}(z^t)\\
&&+(-\frac{p}{2\beta}+\frac{p}{{\sigma}_4})\|z^t-z^{t+1}\|^2+p
\|x(y^{t+1}, z^{t+1})-x^*(z^t)\|^2\nonumber\\
&\le& \Phi_p^{t+1}(z^t)+{\Phi}_d^{t+1}(z^t)+{\Phi}_{prx}^{t+1}(z^t)\\
&&+(-\frac{p}{2\beta}+\frac{p}{{\sigma}_4})\|z^t-z^{t+1}\|^2+\frac{p\sigma_5^2}{\sigma_3^2}\|Ax(y^{t+1}, z^t)-b\|^2,\nonumber
\end{eqnarray*}
where the last step is due to the error bounds \eqref{eb6}, \eqref{eb5} in Lemma \ref{error bound}.
Let
$$
D_1(\beta)=\tau_1,\quad D_2(\beta)=\tau_2+\frac{p\sigma_5^2}{\sigma_3^2},\quad D_3(\beta)=\tau_3,\quad D_4(\beta)=-\frac{p}{2\beta}+\frac{p}{{\sigma}_4},
$$
then we have
\begin{eqnarray*}\nonumber
\Phi^{t+1}(z^{t+1}) &=&\Phi_p^{t+1}(z^{t+1})+\Phi_d^{t+1}(z^{t+1})+\Phi_{prx}^{t+1}(z^{t+1})\\
&\le&\Phi_p^{t+1}(z^t)+{\Phi}_d^{t+1}(z^t)+{\Phi}_{prx}^{t+1}(z^t)\\
&&+(-\frac{p}{2\beta}+\frac{p}{{\sigma}_4})\|z^t-z^{t+1}\|^2+\frac{p\sigma_5^2}{\sigma_3^2}\|Ax(y^{t+1}, z^t)-b\|^2\\
&\le& D_1(\beta)\|x^t-x^{t+1}\|^2+D_2(\beta)\|Ax(y^{t+1}, z^t)-b\|^2\nonumber\\
&&+D_3(\beta)\|z^t-x^*(z^t)\|^2+D_4(\beta)\|z^t-z^{t+1}\|^2,\label{eq:Delta}
\end{eqnarray*}
where the last step follows from \eqref{eq1}-\eqref{eq3}.
Recall the definition of potential function $\phi$ (cf.\ \eqref{phi}). It follows that
\begin{eqnarray*}
\phi^{t}-\phi^{t+1}&=&\left({\Phi}_p^{t+1}(z^t)+{\Phi}_d^{t+1}(z^t)+{\Phi}_{prx}^{t+1}(z^t)\right)
-\left(\Phi_p^{t+1}(z^{t+1})+\Phi_d^{t+1}(z^{t+1})+\Phi_{prx}^{t+1}(z^{t+1})\right)\\
&=&\Phi^t(z^t)-\Phi^{t+1}(z^{t+1})\\
&\ge&C_1(\beta)\|x^t-x^{t+1}\|^2+C_2(\beta)\|Ax(y^{t+1}, z^t)-b\|^2\nonumber\\ &&
+C_3(\beta)\|z^t-x^*(z^t)\|^2+C_4(\beta)\|z^t-z^{t+1}\|^2 \\ 
&\ge&\kappa\left(D_1(\beta)\|x^t-x^{t+1}\|^2+D_2(\beta)\|Ax(y^{t+1}, z^t)-b\|^2\right.\nonumber\\ &&
\left.+D_3(\beta)\|z^t-x^*(z^t)\|^2+D_4(\beta)\|z^t-z^{t+1}\|^2\right)\\
&\ge &\kappa \Phi^{t+1}(z^{t+1}) ,
\end{eqnarray*}
where $\kappa:=\min_{i: D_i(\beta)>0}\{\frac{C_i(\beta)}{D_i(\beta)}\}$ and the first inequality follows from Lemma \ref{descent}. This completes the proof.
\end{proof}

Combining the above analysis, we can prove the following theorem.
\begin{theorem}\label{linear-convergence}
There exists a $t_{\max}>0$ such that
$$\Phi^t(z^t)-\Phi^{t+1}(z^{t+1})\ge \kappa\Phi^{t+1}(z^{t+1}), \ \mbox{ for any }t\ge t_{\max}.$$
\end{theorem}
\begin{proof}
By Claim \ref{claim},  we have
$$f(\bar{z}^t)=f^*,\ \mbox{ when $t$ is sufficiently large}.$$
Moreover, according to Theorem \ref{th:main}, when $t$ is sufficiently large, we have
\begin{eqnarray*}
	\dist(z^t, X^*)&<&\Delta,\\
	\|Ax(y^{t+1}, z^t)-b\|&<&\Delta.
\end{eqnarray*}
Hence there exists a $t_{\max}$ such that the conditions of Lemma~\ref{one-step-linear-convergence} hold after $t\ge t_{\max}$. The linear convergence follows immediately from Lemma~\ref{one-step-linear-convergence}.
\end{proof}

\noindent{\bf Remark.} Note that it follows from Lemma~\ref{proximal residual}, \eqref{potential-descent} and Theorem \ref{linear-convergence} that the sequence $\{z^t\}$ converges R-linearly  after $t\ge t_{\max}$.
The constant $t_{\max}$ depends on the problem structure and the initial point. It can be shown that
$$t_{\max}=\max\left\{\frac{\Phi^0(z^0)\bar{\tau}^2\bar{B}}{\min\{\Delta, \epsilon(x^0, y^0, z^0)\}}, \frac{\Phi^0(z^0)}{\Delta^2\alpha/2}\right\}\cdot \max\left\{1, 1/\kappa\right\}.$$
However, this estimate depends on some unknown constants such as $\Delta$ and $\epsilon(x^0,y^0,z^0)$ which are implicitly defined by the problem structure and geometry. A global error bound is needed to provide an explicit estimate of $t_{\max}$. This will be the focus of a forthcoming paper.

\section{Multi-block Case and a Linearized Proximal ADMM}
In this section we consider the multi-block case. Consider the following optimization problem:
\begin{equation}
\label{P2}
\begin{array}{ll}
\mbox{minimize}& f(x_1, x_2, \cdots, x_k)\\ [5pt]
\mbox{subject to} & \sum_{j=1}^kA_jx_j=b,\ x\in P,
\end{array}
\end{equation}
where $x_j\in \mathbb{R}^{n_j}$ for all $1\le j\le k$, $x=(x_1, x_2, \cdots, x_k)\in \mathbb{R}^n$ and $P=\prod_{i=1}^n[\ell_i, u_i]$.
We also denote $A=[A_1, A_2,\cdots, A_k]$ and $\bar{\sigma}=\max_{1\le i\le k}\|A_i\|_2$, where $\|\cdot\|_2$ is the spectral norm of a matrix.
We still adopt Assumption \ref{feasible-int} and hence every partial gradient $\nabla_{x_i}f(\cdot)$ is $L$-Lipschitz-continuous.
To solve this problem, we use the following linearized proximal ADMM \ref{Alg3}, which updates the primal variables blockwise. Let us denote
$$x^t(j)=(x_1^{t+1}, x_2^{t+1}, \cdots, x_{j-1}^{t+1}, x_j^t, \cdots, x_k^t).$$
\begin{algorithm}[ht]
\caption{Linearized Proximal ADMM}
\label{Alg3}
\begin{algorithmic}[1]
\STATE Let $\alpha>0$, $0<\beta\le 1$ and $0<c\le \frac{1}{\bar{L}+p+\Gamma\bar{\sigma}^2}$;
\STATE Initialize $x^0\in P,\ z^0\in P,\ y^0\in \mathbb{R}^m$;
\FOR{$t=0,1,2,\ldots,$}
\STATE $y^{t+1}=y^t+\alpha(Ax^t-b)$;
\FOR{$j=1, 2, \cdots, k$}
{\color{black}\STATE $x^{t+1}_j=[x^t_j-c\nabla_{x_j} K(x^t(j), z^t; y^{t+1})]_+^j$;}
\ENDFOR
\STATE $z^{t+1}=z^t+\beta(x^{t+1}-z^t)$. 
\ENDFOR
\end{algorithmic}
\end{algorithm}
{\color{black}Here $[\cdot]_+^j$ means the projection to the set $\prod_{i\in {\cal N}_j}[\ell_i, u_i]$, where ${\cal N}_j$ is the set of indexes for the $j$-th variable block $x_j$.} Note that here we take the stepsize $$c<\frac{1}{{L}+p+\Gamma\bar{\sigma}^2}.$$ 
To prove the convergence of Algorithm \ref{Alg3}, we can follow the same line as that for the proof of Theorem \ref{th:main}. The only differences are the proof for \eqref{eb1}-\eqref{eb3} in Lemma \ref{error bound} and the proof of Lemma \ref{primal}.
For the primal error bounds \eqref{eb1}-\eqref{eb3}, we only give the proof of the first one and the others can be proved using the same techniques as that in the proof of Lemma \ref{error bound}.
\begin{lemma}
For any $t$, there exists a constant $\bar{\sigma}_1$, such that
$$\|x^t-x^{t+1}\|\ge \bar{\sigma}_1\|x^t-x(y^{t+1}, z^t)\|.$$
\end{lemma}
\begin{proof}
Since this lemma is not related to the update of $y, z$, for notation simplicity, we denote $K(x)=K(x, z^t; y^{t+1})$.
The proof consists of two parts:
\begin{enumerate}
\item When the primal variables are updated via the block coordinate gradient descent scheme, it is a type of approximate gradient projection algorithm, namely,
$$x^{t+1}=[x^t-c\nabla K(x^t)+\Theta(t)]_+,$$
where $\Theta(t)$ satisfies
\begin{equation}\label{Theta}
\|\Theta(t)\|< \eta\|x^t-x^{t+1}\|
\end{equation}
for some positive constant $\eta>0$.
\item Prove the approximate gradient projection algorithm has the primal error bound.
\end{enumerate}
We consider the first part. Notice that
\begin{eqnarray}
x^{t+1}_j&=&[x^t_j-c\nabla_{x_j} K(x^t(j))]_+^j,\nonumber\\
&=&[x^t_j-c\nabla_{x_j} K(x^t)+c(\nabla_{x_j} K(x^t(j))-\nabla_{x_j} K(x^t))]_+^j\nonumber\\
&=&[x^t-c\nabla_{x_j} K(x^t)+\Theta_j(t)]_+^j,\nonumber,
\end{eqnarray}
where $\Theta_j(t)=c\left(\nabla_{x_j} K(x^t(j))-\nabla_{x_j} K(x^t)\right)$.
Due to the Lipschitz continuity of the partial gradient of $K$, we have
\begin{eqnarray}
\|\Theta_j(t)\|&\le&c({L}+p+\Gamma\bar{\sigma}^2)\|x^t(j)-x^t\|\nonumber\\
&\le&c({L}+p+\Gamma\bar{\sigma}^2)\sqrt{\sum_{i=1}^{j-1}\|x_i^{t+1}-x_i^t\|^2}\nonumber\\
&\le& c({L}+p+\Gamma\bar{\sigma}^2)\sum_{i=1}^{j-1}\|x_i^t-x_i^{t+1}\|,
\end{eqnarray}
where the last inequality is proved by just squaring both sides of the inequality.
Since $\sum_{i=1}^{j-1}\|x_i^t-x_i^{t+1}\|\le \sum_{i=1}^k\|x_i^t-x_i^{t+1}\|$, we have
\begin{eqnarray}
\|\Theta(t)\|&=&\|(\Theta_1(t), \cdots, \Theta_k(t))\|\nonumber\\
&\le&c({L}+p+\Gamma\bar{\sigma}^2)k\sum_{i=1}^k\|x^t_i-x^{t+1}_i\|\nonumber\\
&\le&c({L}+p+\Gamma\bar{\sigma}^2)k^{3/2}\|x^t-x^{t+1}\|,\nonumber
\end{eqnarray}
where the last inequality is due to Cauchy-Schwartz inequality.
This finishes the proof of \eqref{Theta} with $\eta=c({L}+p+\Gamma\sigma^2)k^{3/2}$.
For the second part, we have
\begin{eqnarray}
\|x^t-\!\!&x&\!\!\!^{t+1}\|=\|x^t-[x^{t}-c\nabla K(x^t)+\Theta(t)]_+\|\nonumber\\
&\stackrel{ \mbox{\scriptsize(i)}}\ge& \|x^t-[x^t-c\nabla K(x^t)]_+\|-\|[x^t-c\nabla K(x^t)]_+-[x^t-c\nabla K(x^t)+\Theta(t)]_+\|\nonumber\\
&\stackrel{ \mbox{\scriptsize(ii)}}\ge& \sigma_1\|x^t-x(y^{t+1}, z^t)\|-\|[x^t-c\nabla K(x^t)]_+-[x^t-c\nabla K(x^t)+\Theta(t)]_+\|\nonumber\\
&\stackrel{ \mbox{\scriptsize(iii)}}\ge& \sigma_1\|x^t-x(y^{t+1}, z^t)\|-\|\Theta(t)\|\nonumber\\
&\stackrel{ \mbox{\scriptsize(iv)}}\ge& \sigma_1\|x^t-x(y^{t+1}, z^t)\|-\eta\|x^t-x^{t+1}\|,\nonumber
\end{eqnarray}
where (i) is because of the triangular inequality, (ii) is due to the error bound \eqref{eb1} in Lemma \ref{error bound}, (iii) is due to the nonexpansiveness of the projection operator and (iv) is because of \eqref{Theta}. Hence setting $\bar{\sigma}_1=\sigma_1/(1+\eta)$ completes the proof.
\end{proof}
Next we establish a simple lemma to ensure that the sufficient decrease result Lemma \ref{primal} holds true for the multi-block case.
\begin{lemma}\label{primal2}
For any $t$, we have
$$K(x^t, z^t; y^{t+1})-K(x^{t+1}, z^t; y^{t+1})\ge \frac{1}{2c}\|x^t-x^{t+1}\|^2.$$
\end{lemma}
\begin{proof}
Notice that, by Assumption~\ref{feasible-int}(c), the partial gradient of $K$ with respect to any block is $c^{-1}$-Lipschitz continuous, so we have
$$K(x^t(j), z^t; y^{t+1})-K(x^t(j+1), z^t ;y^{t+1})\ge \frac{1}{2c}\|x^t_j-x^{t+1}_j\|^2,\mbox{ for all }1\le j\le k.$$
Here $x^t(0)=x^t$.
Summing this from $0$ to $k-1$ and using the fact that
$$\sum_{j=1}^k\|x^t_j-x^{t+1}_j\|^2=\|x^t-x^{t+1}\|^2$$
yields the desired result.
\end{proof}
This shows that the descent condition \eqref{eq:descent} holds true for the multi-block case, which further implies Lemma~\ref{primal} remains valid.
Equipped with these, we conclude that the Algorithm \ref{Alg3} converges globally.
\begin{theorem}
Suppose Assumption \ref{feasible-int} holds and the parameters $c$ and $p$ satisfy
$$\frac{1}{{L}+p+\Gamma\bar{\sigma}^2}>c>0,\quad p>-\gamma$$
and that the primal and dual stepsizes $\beta$ and $\alpha$ are sufficiently small. Then the dual iterates $\{y^t\}$ are bounded. Moreover, there holds
\begin{eqnarray}
\lim_{t\to \infty}\|x^{t+1}-x^t\|&=&0,\nonumber\\
\color{black}\lim_{t\to \infty}\dist(x^t, X^*)&=&0,\nonumber\\
\lim_{t \to \infty}\dist(z^t, X^*)&=&0,\nonumber
\end{eqnarray}
and every limit point of the sequence $\{(x^t, y^t)\}$ generated by Algorithm \ref{Alg3} is a primal-dual stationary point of \eqref{P2}.
\end{theorem}
\section{Numerical Results}
In this section, we show some numerical results to show that our algorithm is more efficient than some existing algorithms and the exponential weighting step is helpful for convergence.
We consider the quadratic programming problem:
\begin{equation}
\begin{aligned}
(QP)\quad \min_{\bm{x}}\ &\frac{1}{2}\bm{x}^T\bm{Q}\bm{x}+\bm{r}^T\bm{x}\\
\textrm{s.t.} \ & \bm{Ax}=\bm{b},\\
& \ell_i\leq \bm{x}_i \leq u_i,
\end{aligned}
\end{equation}
where $\bm{x}\in\mathbb{R}^n$, $\bm{Q}\in\mathbb{R}^{n\times n}$, $\bm{r}\in\mathbb{R}^n$, $\bm{A}\in\mathbb{R}^{m\times n}$, and $\bm{b}\in\mathbb{R}^m$.
\subsection{Why do we need the auxiliary sequence $\{z^t\}$?}
A natural question is whether we can set $\beta=1$ in Algorithm~\ref{Alg2} and thus eliminating the sequence $\{z^t\}$ from the iterations. In this subsection, we give some numerical results showing that $\beta<1$ is needed for convergence
and hence the sequence $\{z^t\}$ is necessary in Algorithm~\ref{Alg2}.
Here we set $n=500, m=100$. $Q=-U^TU$, where $U$ is a $500\times 500$ matrix with each entry following the normal distribution $N(0, 1)$. Every entry of $A$ and $r$ is generated from the normal distribution $N(0, 1)$. $b=Ax_0$, with $x_0\in \mathbb{R}^{500}$ uniformly distributed over $[0,5]^{500}$.Moreover,  $\ell_i=0$ and $u_i=1000$ for all $1\le i\le 500$.

In our experiment, we set $\Gamma=1000$ and $p=5000$. Consider the following two cases.

\noindent{\bf Case 1.} We use $\alpha=1000, 50, 1$ and plot the curves for $\|Ax^t-b\|$ in Fig.~\ref{fig:a}.

We see that Algorithm \ref{Alg} oscillate after $2\times 10^6$ iterations for these choices of $\alpha$.\\
\noindent{\bf Case 2.} We choose $\alpha=50$, $\beta=1, 0.02, 0.01$ and plot $f(x^t), \|Ax^t-b\|, \|x^{t+1}-z^t\|$ and $\frac{1}{c}\|x^t-x^{t+1}\|$. This is shown in Fig.~\ref{fig:b}, Fig.~\ref{fig:c} and Fig.~\ref{fig:d}.

We see that Algorithm \ref{Alg2} converges with $\beta=0.02$ and $\beta=0.01$. The algorithm with $\beta=0.02$ is faster, which suggests that we can try larger $\beta$ to achieve higher convergence speed.
  \begin{figure}[H]
  \centering
    \includegraphics[width=7.5cm,height=7.5cm]{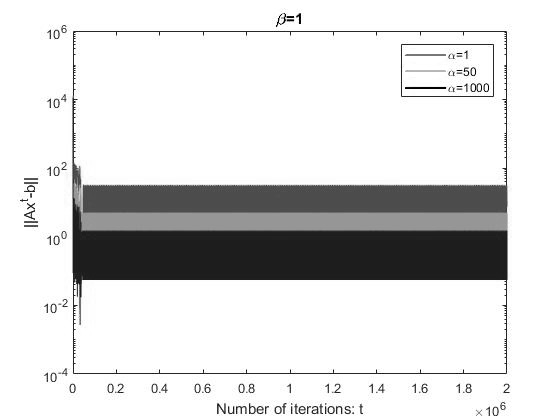}
    \caption{$\|Ax^t-b\|$ v.s.\ iterations with $\beta=1$}
    \label{fig:a}

  \centering
    \includegraphics[width=7.5cm,height=7.5cm]{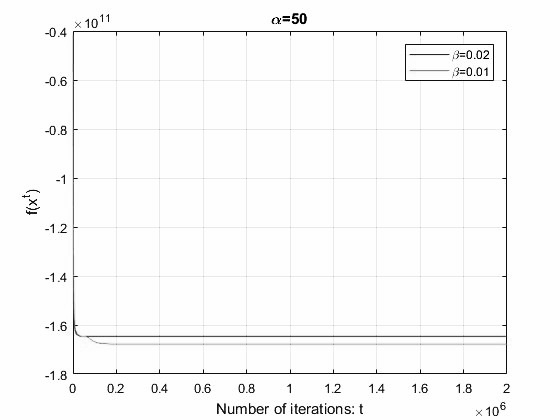}
    \caption{$f(x^t)$ v.s.\ iterations}
    \label{fig:b}
  \end{figure}
  \begin{figure}[H]
   \centering
    \includegraphics[width=7.5cm,height=7.5cm]{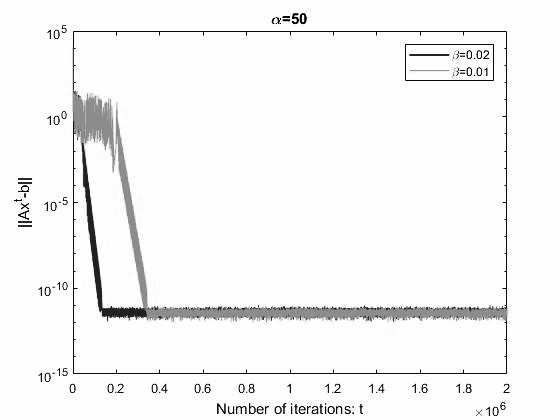}
    \caption{$\|Ax^t-b\|$ v.s.\ iterations}
    \label{fig:c}

  \centering
    \includegraphics[width=7.5cm,height=7.5cm]{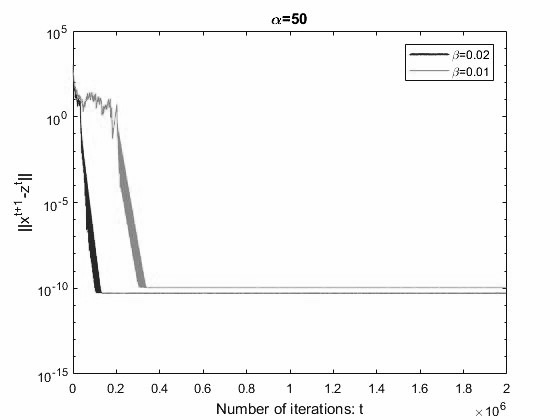}
    \caption{$\|x^{t+1}-z^t\|$ v.s.\ iterations}
    \label{fig:d}
\end{figure}


\subsection{Comparison with the penalty method}

We compare our proximal-ADMM algorithm with the QP-AIPP algorithm of \cite{Monteiro}  for solving the above QP with $n=20,m=5$, where $\bm{Q},\bm{A},\bm{r},\bm{b}$ are randomly generated from uniform distribution $\mathcal{U}[0,1]$, and $\ell_i$ and $u_i$ are set to be $0$ and $1$ respectively.

Parameters for our algorithm are set as follows:
$\Gamma=10$, $\alpha=\Gamma/4$, $p=2L_Q+2\Gamma\sigma^2$, $c=1/(2(L_Q+\Gamma\sigma^2))$, and $\beta=0.5$, where $L_Q$ is the eigenvalue of $Q$ with largest absolute value.

Parameter setting of the QP-AIPP algorithm are as follows (notations in Section 3.2 of \cite{Monteiro}, AIPP Algorithm): $\sigma=0.1$, $\bar{\Gamma}=10^{-3}$, $\bar{\epsilon}=10^{-3}$, $c=10^4$, $L=c\| \bm{A} \|_F^2+L_Q$ ($L_Q$ is defined above), and $\lambda=1/(2L_Q)$.

We use the sum of first order optimality gap and feasibility gap as a measure to evaluate the convergence behavior of the two algorithms, i.e., $\|x^t-[x^t-\nabla_xL(x^t; \lambda^t)]_+\|+\|Ax^t-b\|$.
Here for our algorithm, we set $\lambda^t=y^t$ while for the penalty method in \cite{Monteiro}, $\lambda^t=C(Ax^t-b)$.
We randomly generated two examples and the residuals are plotted in Fig.~\ref{fig:exp}. 
\begin{figure}[H]
\centering

\includegraphics[width=7.5cm,height=7.5cm]{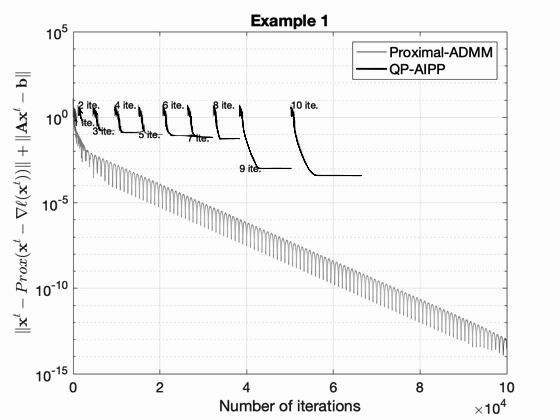}
\centering
\includegraphics[width=7.5cm,height=7.5cm]{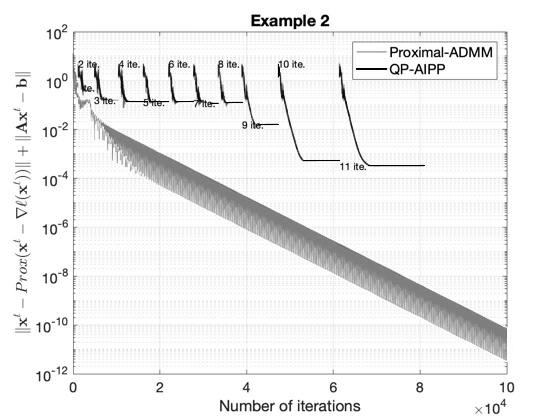}
\caption{Comparison of two examples.}
\label{fig:exp}
\end{figure}

\subsection{Comparison with the traditional ADMM}
Lastly, we show that the algorithm with linearization steps is faster than the traditional ADMM where any sub-problem is solved to a high accuracy by a first-order algorithm.

We consider the 2-block case:
\begin{equation}\nonumber
\begin{array}{rl} \min & x_1^TQ_1x_1+x_2^TQ_2x_2\\
\textmd{s.t.} & A_1x_1+A_2x_2=b,\\
 & x_1\in P_1, x_2\in P_2.
\end{array}
\end{equation}
Here $Q_1$ and $Q_2$ are $(n/2)\times (n/2)$ symmetric matrices, where  each entry
 is generated from a uniform distribution over $[0,1]$. $A_1, A_2\in \mathbb{R}^{m \times (n/2)}$ and $b\in \mathbb{R}^{m}$ with each entry generated from the uniform distribution  over $[0,1]$. $P_1=P_2=[0, 10]^{n/2}$ are two boxes.
The following shows the convergence of the original ADMM and our smoothed proximal ADMM using the same measure as the above subsections.
\begin{table}[htbp]
	\centering
	\caption{Original ADMM v.s.\ the proposed Proximal ADMM (number of gradient evaluations)}
	\label{table1}

	\begin{tabular}{|c|c|c|c|c|}
		\hline
		$n$&$m$&eps&Original ADMM &the proposed Proximal ADMM \\
		\hline
		20&2&$10^{-4}$&20695&852 \\
        20&8&$10^{-4}$&86359&1024 \\
        20&2&$10^{-5}$&136495&7845 \\
        20&8&$10^{-5}$&162870&11743 \\
		\hline
	\end{tabular}
\end{table}
This experiment shows that, as a single-loop algorithm, the proposed Proximal ADMM algorithm is more efficient than the original double-loop ADMM algorithm in which the inner-loop can be time consuming.

\appendix



\section{Proof of Corollary \ref{coro1}}\label{Appendix:B}
\begin{proof}
We prove by contradiction.
Suppose that there exist sequences $\{y^{i}\}\subseteq \tilde{y}^0+\mathrm{range}(A),  \subseteq (\tilde{y}^0+\mathrm{Range}(A))$ and $\{z^i\}$ such that
$$\lim_{i\to\infty}\|x^i-(x^i)^+\|=\lim_{i\to\infty} \|z^i-(x^i)^+\|=\lim_{i\to\infty} \|Ax((y^i)^+, z^i)-b\|=0$$
while $\max\{\lim\sup_{i\to\infty}\dist(x^i,X^*), \lim\sup_{i\to\infty}\dist(z^i, X^*)\}>0$.
Passing to a sub-sequence if necessary, we assume that $\max\{\lim_{i\to \infty}\dist(x^i ,X^*), \lim_{i\to\infty}\dist(z^i, X^*),$ $\lim_{i\to \infty}\|x((y^i)^+, z^i)-x^*(z^i)\|\}>0$.
On one hand, since $\{y^i\}\subseteq \tilde{y}^0+\mathrm{Range}(A)$ and $\|Ax(y^i, z^i)-b\|\rightarrow 0$,  $\{y^i\}$ is bounded according to  (cf.\ Lemma~\ref{ybounded}). By further passing to a subsequence if necessary, we can assume that
$$x^i\rightarrow \bar{x},\quad z^i\rightarrow \bar{z},\quad (y^i)^+\rightarrow \bar{y}.$$
Then we have $\bar{x}-\bar{x}^+=0,\quad \bar{z}-\bar{z}^+=0.$\\
Therefore $\bar{x}=x(\bar{y}, \bar{z})$.
It follows from Lemma \ref{continuity} that
$Ax(\bar{y}, \bar{z})-b=0$, which further implies
$$(\bar{x}, \bar{y}, \bar{z})=(\bar{x}^+, \bar{y}^+, \bar{z}^+).$$
In view of Lemma~\ref{limitpoint}, we obtain that $\bar{x}=\bar{z}\in X^*,\ \bar{z}\in X^*$ and $(\bar{x}, \bar{y})\in W$.
Hence, we have $\dist(\bar{x}, X^*)=\dist(\bar{z}, X^*)=0$
On the other hand, we have $\max\{\dist(\bar{x}, X^*),$ $\dist(\bar{z}, X^*)\}>0$. A contradiction!
\end{proof}


\section{Proof of Corollary \ref{coro2}}\label{Appendix:C}
\begin{proof}
Again we prove by contradiction.
Suppose the contrary so that there exists sequences $\{y^i\}$ and $\{z^i\}$ such that $\|Ax((y^i)^+, z^i)-b\|\rightarrow 0$ and $\lim\sup_{i\rightarrow \infty}\|x((y^i)^+,\!z^i)$ $-x^*(z^i)\|>0$.
Passing to a sub-sequence if necessary, we can assume  without loss of generality that $\lim_{i\rightarrow \infty}\|x((y^i)^+, z^i)-x^*(z^i)\|>0$.
By Lemma \ref{ybounded}, because $\|Ax((y^i)^+, z^i)-b\|\rightarrow 0$, \{$(y^i)^+\}$ is bounded. Hence, there exists at least one limit point of $\{(y^i)^+\}$. Therefore, further passing to a sub-sequence if necessary, we assume that $(y^i)^+\rightarrow \bar{y}$ and $z^i\rightarrow \bar{z}\in P$.
Combining the above analysis and using Lemma \ref{continuity}, we have
$$\|Ax(\bar{y}, \bar{z})-b\|=0.$$
It follows from the KKT condition for \eqref{K} at $z^i$ that $\bar{y}$ (the limit of $(y^i)^+$) is the optimal dual multiplier for the problem:
\begin{equation*}
\begin{array}{ccc}
\mbox{minimize} & \displaystyle f(x)+\frac{p}{2}\|x-\bar{z}\|^2 \\ [5pt]
\mbox{subject to} & Ax=b,\ x\in P,
\end{array}
\end{equation*}
implying $ x(\bar{y}, \bar{z})=x^*(\bar{z})$. This is a contradiction.
\end{proof}
\section{Proof of Inequalities \eqref{eb1}-\eqref{eb3} in Lemma \ref{error bound}}\label{Appendix:D}
\begin{proof}
We first prove \eqref{eb1}.
By the definition of $K(x,y;z)$ (cf.\;\eqref{eq:K}) and Assumption~\ref{feasible-int}(b), $\nabla K(x,z^t;y^{t+1})$ is Lipschitz continuous in $x$
\[
\|\nabla_x K(x^t,z^t;y^{t+1})-\nabla_x K(x(y^{t+1},z^t),z^t;y^{t+1})\|\le (p+L+\Gamma\sigma^2)\|x^t-x(y^{t+1},z^t)\|.
\]
So the Lipschitz constant for $c\nabla_xK$ is $c(L+p+\Gamma\sigma^2)$, where $\sigma$ is the spectral norm of $A$.
Since $\gamma_K=p+\gamma>0$, it follows that $cK(x,z^t;y^{t+1})$ is strongly convex in $x$ with modulus $c(p+\gamma)$, so from \cite{[22]} that the following global error bound holds
\begin{equation}\label{geb}
\|x-[x-c\nabla K(x,z^t;y^{t+1})]_+\|\ge \sigma_1\|x-x(y^{t+1},z^t)\|, \quad \forall x\in \mathbb{R}^n,
\end{equation}
where $\sigma_1=c\gamma_K$.
Specializing \eqref{geb} at $x=x^t$ and noticing that
$$[x^t-c\nabla K(x^t,z^t;y^{t+1})]_+=x^{t+1},$$
we obtain \eqref{eb1}.

By the triangular inequality, we further obtain
\begin{eqnarray*}
\|x^{t+1}-x(y^{t+1},z^t)\|
&\le& \|x^t-x^{t+1}\|+\|x^{t}-x(y^{t+1}, z^t)\|\nonumber\\
&\le& \|x^t-x^{t+1}\|+\frac{1}{\sigma_1}\|x^t-x^{t+1}\|\nonumber\\
&\le&(1+\frac{1}{\sigma_1})\|x^t-x^{t+1}\|.\nonumber
\end{eqnarray*}
This shows that \eqref{eb2} holds with $\sigma_2=\sigma_1/(1+\sigma_1)$.
\end{proof}

\bibliography{references}
\bibliographystyle{plain}
\end{document}